\documentclass{amsart}
\usepackage{latexsym}
\usepackage{amsmath}
\usepackage{amssymb}
\usepackage{amsthm}
\usepackage{amscd}
\numberwithin{equation}{section}

\newtheorem{theorem}{Theorem}[section]
\newtheorem{lemma}[theorem]{Lemma}
\newtheorem{corollary}[theorem]{Corollary}
\newtheorem{proposition}[theorem]{Proposition}

\theoremstyle{definition}

\newtheorem{remark}[theorem]{Remark}
\newtheorem{definition}[theorem]{Definition}
\newtheorem{example}[theorem]{Example}
\newtheorem{examples}[theorem]{Examples}
\theoremstyle{remark}
\newtheorem{notation}{Notation}

\begin{document}

\title{Homogeneous Besov spaces}

\author{Yoshihiro Sawano}

\maketitle

\begin{abstract}
This note is based on a series of lectures
delivered in Kyoto University.
This note surveys
the homogeneous Besov space
$\dot{B}^s_{pq}$
on ${\mathbb R}^n$ 
with $1 \le p,q \le \infty$ and $s \in {\mathbb R}$
in a rather self-contained manner.
Possible extensions
of this type of function spaces are breifly discussed
in the end of this article.
In particular,
the fundamental properties
are stated for the spaces
$\dot{B}^s_{pq}$
with $0<p,q \le \infty$ and $s \in {\mathbb R}$
and
$\dot{F}^s_{pq}$
with $0<p<\infty$,
$0<q \le \infty$ and $s \in {\mathbb R}$
as well as nonhomogeneous coupterparts
$B^s_{pq}$
with $0<p,q \le \infty$ and $s \in {\mathbb R}$
and
$F^s_{pq}$
with $0<p<\infty$,
$0<q \le \infty$ and $s \in {\mathbb R}$.
\end{abstract}

{\bf 2010 Classification}:
42B35,
46A04

\section{Introduction}

In this note we mean by a \lq \lq function space"
a linear function space made up of functions
especially defined on ${\mathbb R}^n$.
We envisage the following function spaces;
$C^k$,
$C^\infty$,
$C^\infty_{\rm c}$,
${\rm BC}$
(the set of all bounded continuous functions),
${\mathcal O}({\mathbb C})$
(the set of all holomorphic functions on ${\mathbb C}$
under the cannonical homeomorphism ${\mathbb C} \sim {\mathbb R}^2$)
and
the Sobolev space $W^{m,p}$
e.t.c..
What type of functions do these function spaces collect?
\begin{enumerate}
\item
Those which deal with size of functions:
$L^p$, ${\rm BC}$, $W^{m,p}$
\item
Those which deal with differentiability of functions:
$C^k$,
$C^\infty$,
$C^\infty_{\rm c}$,
${\rm BC}$ ($0$-times differentiability), $W^{m,p}$,
${\mathcal O}({\mathbb C})$.
\end{enumerate}
The Sobolev spaces can handle both of them
and are equipped with two parameters.

We can say that the function spaces can describe
many properties if they have many parameters.
This in turn implies that the more properties they can
describe, the more complicated their definition is.

\subsection{Advantage of Besov spaces}

Generally speaking, the function spaces
are difficult to handle if their definition is complicated.
However, although the definition of Besov spaces
is complicated, Besov spaces are important in many senses.

\subsubsection{\bf Mathematical transforms}

Besov spaces grasp many mathematical transforms nicely.
Let us give examples
without giving the precise definition of Besov spaces.
\begin{theorem}[Riemann-Lebesgue]\label{thm:1}
Define the Fourier transform ${\mathcal F}$
by
\[
{\mathcal F}f(\xi)\equiv
\frac{1}{\sqrt{(2\pi)^n}}
\int_{{\mathbb R}^n}f(x)e^{-i x \xi}\,dx
\quad (\xi \in {\mathbb R}^n).
\]
Then we have
${\mathcal F}$ maps $L^1$ continuously to
$L^\infty$.
\end{theorem}
Theorem \ref{thm:1} is known as the Riemann-Lebesgue theorem.
However, in terms of the (nonhomogeneous) Besov spaces
$B^0_{1\infty}$ and $B^0_{\infty 1}$,
we can refine Theorem \ref{thm:1} as follows:
\begin{theorem}\label{thm:2}
\
\begin{enumerate}
\item
${\mathcal F}$ maps $L^1$ continuously to
$B^0_{\infty 1}$.
\item
${\mathcal F}$ maps $B^0_{1\infty}$ continuously to
$L^\infty$.
\end{enumerate}
\end{theorem}
Since
$B^0_{\infty1}$ is embedded into $L^\infty$ continuously
and
$B^0_{1\infty}$ is embedded into $L^1$ continuously,
Theorem \ref{thm:2} refines Theorem \ref{thm:1}
in two different directions.
We prove Theorem \ref{thm:2} in Section \ref{s3.5}.

Let us give another example.
\begin{theorem}\label{thm:151224-1}
Let $1 \le p,q \le \infty$ and $s>0$.
For $j=1,2,\ldots,n$,
we define the $j$-th Riesz transform by:
\[
R_j f(x) \equiv
\lim_{\varepsilon \to 0}
\int_{{\mathbb R}^n \setminus B(x,\varepsilon)}
\frac{(x_j-y_j)f(y)}{|x-y|^{n+1}}\,dy
\quad (x \in {\mathbb R}^n).
\]
Then 
$R_j$ maps $\dot{B}^s_{pq}$ continuously to itself.
\end{theorem}
This is a big achievement;
$R_j$ does not map $L^1$ into $L^1$ continuously.
Since $\dot{B}^0_{12}$ is almost the same as $L^1$,
$\dot{B}^0_{12}$ nicely substitutes for $L^1$.
We prove Theorem \ref{thm:151224-1} in Section \ref{Section:A3}.

\subsubsection{\bf Special cases}
Many important function spaces are realized
as special cases of Besov spaces and Triebel-Lizorkin spaces,
which are variants of Besov spaces.
We refer to 
Sections \ref{section:4.1}
and
\ref{section:4.2}
for the definition of 
the {\it homogeneous Besov space} $\dot{B}^s_{pq}$
with
$0<p,q \le \infty$
and 
$s \in {\mathbb R}$
and
the {\it homogeneous Triebel-Lizorkin space} $\dot{F}^s_{pq}$
with
$0<p<\infty$,
$0<q \le \infty$
and 
$s \in {\mathbb R}$,
while we shall recall the definition of 
the {\it nonhomogeneous Besov space} $B^s_{pq}$
with
$0<p,q \le \infty$
and 
$s \in {\mathbb R}$
and
the {\it nonhomogeneous Triebel-Lizorkin space} 
$F^s_{pq}$
with
$0<p<\infty$,
$0<q \le \infty$
and 
$s \in {\mathbb R}$ 
in Sections
\ref{section:6.3}
and
\ref{section:6.4},
respectively.

\begin{theorem}\label{thm:160118-5}
Denote by $\{e^{t\Delta}\}_{t>0}$
the heat semi-group.
\begin{enumerate}
\item
For $1<p<\infty$
\begin{equation}\label{eq:160118-1}
L^p \sim \dot{F}^0_{p2} \sim F^0_{p2}.
\end{equation}
\item
$H^p \sim \dot{F}^0_{p2}$
for $0<p<\infty$,
where $H^p$ denotes the Hardy space
of all distributions $f$ for which
\begin{equation}\label{eq:160118-2}
\|f\|_{H^p} \equiv
\left\|\sup_{t>0}|e^{t\Delta}f|\right\|_{p}<\infty
\quad (f \in {\mathcal S}').
\end{equation}
\item
$h^p \sim F^0_{p2}$
for $0<p<\infty$,
where $h^p$ denotes the local Hardy space
of all distributions $f$ for which
\begin{equation}\label{eq:160118-3}
\|f\|_{h^p} \equiv
\left\|\sup_{t \in (0,1)}|e^{t\Delta}f|\right\|_{p}<\infty
\quad (f \in {\mathcal S}').
\end{equation}
\item
Denote by ${\rm BMO}$ the space
of bounded mean oscillations whose norm is given by:
\[
\|f\|_{{\rm BMO}}
\equiv
\sup_{B:{\rm balls}}
\frac{1}{|B|}\int_B 
\left|f(x)-\frac{1}{|B|}\int_B f(y)\,dy\right|\,dx
\]
for $f \in L^1_{\rm loc}$.
Then we have
\begin{equation}\label{eq:160118-4}
{\rm BMO} \sim \dot{F}^0_{\infty 2},
\end{equation}
\item
Denote by
${\rm bmo}$ the space
of local bounded mean oscillations whose norm is given by:
\begin{align*}
\|f\|_{{\rm bmo}}
&\equiv
\sup_{B:{\rm balls}, |B| = 1}
\int_B|f(x)|\,dx\\
&\quad+
\sup_{B:{\rm balls}, |B| \le 1}
\frac{1}{|B|}
\int_B\left|f(x)-\frac{1}{|B|}\int_B f(y)\,dy\right|\,dx.
\end{align*}
Then we have
\begin{equation}\label{eq:160118-5}
{\rm bmo} \sim F^0_{\infty 2}.
\end{equation}
\end{enumerate}
\end{theorem}

\subsubsection{\bf Quantity and quality of functions}

We can easily grasp
the meaning of the parameters $p$ and $s$
in Besov spaces and Triebel-Lizorkin spaces.
As a result,
we can deal with 
the quanitity of functions
and 
the quality of functions
separately.

\subsection{Homogeneous and nonhomogeneous spaces}

Next, we move on to the 
homogeneous spaces
and the 
nonhomogeneous spaces.
Roughly speaking,
homogeneous spaces are function spaces
described by a set of partial derivatives
of the same order;
otherwise the space is nonhomogeneous.

\begin{example}
Let $m \in {\mathbb N}$ and $1 \le p \le \infty$.
\begin{enumerate}
\item
The {\it homogeneous Sobolev norm}
$\displaystyle
\|f\|_{\dot{W}^{m,p}}
\equiv
\sum_{|\alpha|=m}\|\partial^\alpha f\|_p
$
is homogeneous.
\item
The {\it nonhomogeneous Sobolev norm}
$\displaystyle
\|f\|_{W^{m,p}}
\equiv
\sum_{|\alpha| \le m}\|\partial^\alpha f\|_p
$
is nonhomogeneous.
\end{enumerate}
\end{example}

Concerning the homogeneous norms,
a couple of helpful remarks may be in order.
\begin{remark}
Since
differentiation annihilates the polynomials
or decreases the order of the polynomials,
the homogenous norms do not have complete
information of the function $f$.
\end{remark}

Despite this remark,
we have the following good properties.
\begin{enumerate}
\item
We can not use the nonhomogeneous norm
to describe some properties of functions.
For example,
the dilation $f \mapsto f(t\cdot)$ is a typical one.
\item
Although the matters depends
on the equations we consider,
some invariant quantities can be realized
as a sum of homogeneous norms.
We need to handle each term elaborately.
\end{enumerate}

\subsection{Notation}

We use the following notation in this note.

\notation\
\begin{enumerate}
\item
The metric ball defined by $\ell^2$ is usually called a {\it ball}.
We denote by $B(x,r)$
{\it the ball centered at $x$ of radius $r$}.
Given a ball $B$,
we denote by $c(B)$ its {\it center} and by $r(B)$ its {\it radius}.
We write $B(r)$ instead of $B(0,r)$,
where $0 \equiv (0,0,\ldots,0)$.
\item
Let $E$ be a measurable set.
Then, we denote its {\it indicator function}
by $\chi_E$. \index{chi@$\chi$}
If $E$ has positive measure and $E$ is integrable over $f$,
Then denote by $m_E(f)$ the {\it average of $f$} over $E$.
The symbol $|E|$ denotes the {\it volume of $E$}.
\item
Let $A,B \ge 0$.
Then, $A \lesssim B$ and $B \gtrsim A$ mean
that there exists a constant $C>0$
such that $A \le C B$,
where $C$ depends only on the parameters
of importance.
The symbol $A \sim B$ means
that $A \lesssim B$ and $B \lesssim A$
happen simultaneously.
While $A \simeq B$ means that there exists
a constant $C>0$ such that $A=C B$.
\item
We define
\begin{equation}
{\mathbb N}\equiv \{1,2,\ldots\}, \quad
{\mathbb Z}\equiv \{0,\pm 1, \pm 2,\ldots\}, \quad
{\mathbb N}_0\equiv \{0,1,\ldots\}.
\end{equation}
\item
For $a \in {\mathbb R}^n$,
we write
$\langle a \rangle\equiv \sqrt{1+|a|^2}$.
\item
Suppose
that $\{f_j\}_{j=1}^\infty$ is a sequence of measurable functions.
Then we write
\begin{equation}
\| f_j \|_{L^p(\ell^q)}
\equiv
\left(
\int_{{\mathbb R}^n}
\left(
\sum_{j=1}^\infty |f_j(x)|^q
\right)^\frac{p}{q}
\,d x
\right)^\frac{1}{p}, \quad
0< p,q \le \infty
\end{equation}
and
\begin{equation}
\| f_j \|_{\ell^q(L^p)}
\equiv
\left(
\sum_{j=1}^\infty
\left(
\int_{{\mathbb R}^n}
|f_j(x)|^p
\,d x
\right)^\frac{q}{p}
\right)^\frac{1}{q}, \quad
0< p,q \le \infty
\end{equation}
\item
A distribution $f \in {\mathcal S}'$
is said to belong to $L^1_{\rm loc}$
if
there exist constants $C>0$ and $N \in {\mathbb N}$
and $F \in L^1_{\rm loc}$
such that
\[
|\langle f,\varphi \rangle|
=
\left|\int_{{\mathbb R}^n}F(x)\varphi(x)\,dx\right|
\le C
\sum_{|\alpha| \le N}
\sup_{x \in {\mathbb R}^n}
(1+|x|)^N|\partial^\alpha \varphi(x)|
\]
for all $\varphi \in C^\infty_{\rm c}$.
In this case one writes
$f \in L^1_{\rm loc} \cap {\mathcal S}'$,
$F \in L^1_{\rm loc} \cap {\mathcal S}'$,
$f \in {\mathcal S}' \cap L^1_{\rm loc}$
or
$F \in {\mathcal S}' \cap L^1_{\rm loc}$.
\end{enumerate}

\section{Schwartz spaces}

\subsection{Definitions--as sets and as topological spaces}

\subsubsection{\bf The Schwartz space ${\mathcal S}$ and its dual ${\mathcal S}'$}

Here we recall the {\it Schwartz space}
${\mathcal S}$ together with its topology.

\begin{definition}
The {\it Schwartz space}
${\mathcal S}$ is the subspace
of $C^\infty$ given by
$$\displaystyle
{\mathcal S}
=\bigcap_{N=1}^\infty
\left\{f \in C^\infty\,:\,
\sum_{|\alpha| \le N}\sup_{x \in {\mathbb R}^n}
\langle x \rangle^N|\partial^\alpha f(x)|
<\infty \right\}.
$$
The topology of ${\mathcal S}$
is the weakest one for which
the mapping
$f \in {\mathcal S} \mapsto p_N(f) \in {\mathbb R}$
is continuous for all $N \in {\mathbb N}$,
where
\[
p_N(f) \equiv
\sum_{|\alpha| \le N}\sup_{x \in {\mathbb R}^n}
\langle x \rangle^N|\partial^\alpha f(x)|, \quad
\langle x \rangle \equiv \sqrt{1+|x|^2}.
\]
\end{definition}

We admit the following fact:
\begin{theorem}{\rm \cite[p. 249, Th\'{e}or\`{e}me XII]{Schwartz-text}}
\label{thm:160118-4}
The Fourier transform
maps ${\mathcal S}$ into ${\mathcal S}$
isomorphically.
\end{theorem}

The proof is well known and omitted.

Now we move on to the topological dual space
${\mathcal S}'$.
\begin{definition}
One defines
\[
{\mathcal S}'\equiv 
\{f:{\mathcal S} \mapsto {\mathbb C}\,:\,
f\mbox{ is linear and continuous }\}.
\]
One equips
${\mathcal S}'$
with the weakest topology
so that the mapping
\[
f \in {\mathcal S}' \mapsto \langle f,\varphi \rangle
\in {\mathbb C}
\]
is continuous for all $\varphi \in {\mathcal S}$.
\end{definition}

We recall the following fundamental characterization
of ${\mathcal S}'$ for later considerations.

\begin{lemma}\label{lem:151211-1}
For all $f \in {\mathcal S}'$
there exists $N \in {\mathbb N}$ such that
\[
|\langle f,\varphi \rangle| \le N p_N(\varphi)
\]
for all $\varphi \in {\mathcal S}$.
\end{lemma}

\begin{proof}
Since $f$ is continuous at $0$,
\[
\{\varphi \in {\mathcal S}\,:\,|\langle f,\varphi \rangle|<1\}
=
f^{-1}(\{z \in {\mathbb C}\,:\,|z|<1\})
\]
is an open set.
Thus there exists $N \in {\mathbb N}$
and $\delta>0$ such that
\[
\left\{\varphi \in {\mathcal S}\,:\,
p_N(\varphi)<\frac{1}{N}\right\}
\subset
\{\varphi \in {\mathcal S}\,:\,
|\langle f,\varphi \rangle|<1\},
\]
implying
$|\langle f,\varphi \rangle| \le N p_N(\varphi)$
for all $\varphi \in {\mathcal S}$.
\end{proof}

\subsubsection{\bf The Schwartz space ${\mathcal S}_\infty$ and its dual ${\mathcal S}_\infty'$}

The space ${\mathcal S}$ and its dual ${\mathcal S}'$
are the ingredients of
the nonhomogeneous Besov space
$B^s_{pq}$ for $0<p,q \le \infty$ and $s \in {\mathbb R}$,
while for the homogeneous Besov space $B^s_{pq}$ 
with $0<p,q \le \infty$ and $s \in {\mathbb R}$
we need ${\mathcal S}_\infty$
and its dual ${\mathcal S}_\infty'$.

\begin{definition}
One defines
$\displaystyle
{\mathcal S}_\infty
\equiv
\bigcap_{\alpha \in {\mathbb N}_0{}^n}
\left\{
\psi \in {\mathcal S}\,:\,
\int_{{\mathbb R}^n}x^\alpha\psi(x)\,dx=0
\right\}.
$
\end{definition}

The main advantage of defining the class
${\mathcal S}_\infty$ is that
when we are given $\varphi \in {\mathcal S}_\infty$
the function given by
\[
\psi
={\mathcal F}^{-1}[|\xi|^\alpha \cdot {\mathcal F}\varphi]
\]
is in ${\mathcal S}_\infty$.
In fact, for $\varphi \in {\mathcal S}$
$\varphi \in {\mathcal S}_\infty$
if and only if
${\mathcal F}\varphi$ vanishes
up to the arbitrary order at $0$.

\begin{definition}
\
\begin{enumerate}
\item
Equip
${\mathcal S}'_\infty$
with the weakest topology 
so that the evaluation
\[
f \in {\mathcal S}' \mapsto \langle f,\varphi \rangle
\in {\mathbb C}
\]
is continuous for all $\varphi \in {\mathcal S}_\infty$.
\item
Denote by ${\mathcal P}$
the set of all polynomials.
Embed ${\mathcal P}$
into ${\mathcal S}'$ cannonically;
for any $\alpha$ the mononomial
$x^\alpha$ stands for the distribution
\[
\varphi \in {\mathcal S}
\mapsto
\int_{{\mathbb R}^n} x^\alpha \varphi(x)\,dx \in {\mathbb C}.
\]
\end{enumerate}
\end{definition}

Recall that we can endow the quotient $X/\sim$
of a topological space $(X,{\mathcal O}_X)$
and its equivalence relation $\sim$ with a natural topology.

\begin{definition}
Let $(X,{\mathcal O}_X)$ be a topological space.
Denote by $\sim$ the equivalence relation;
\begin{enumerate}
\item
$a \sim a$,
\item
$a \sim b$ implies $b \sim a$,
\item
$a \sim b$ and $b \sim c$ imply $a \sim c$.
\end{enumerate}
\end{definition}

Let $p$ be a projection
$X$ to $X/\sim$.
The strongest topology of $X/\sim$
for which $p$ is continuous is called
the {\it quotient topology} of $X/\sim$.
Remark that 
${\mathcal S}'/{\mathcal P}$ is a quotient space;
$[f]=[g]$ for $f,g \in {\mathcal S}'$
if and only if
$f-g \in {\mathcal P}$.

\subsection{A fundamental theorem}

We shall prove the following theorem:
\begin{theorem}\label{thm:151209-1}
\
\begin{enumerate}
\item
For all $f \in {\mathcal S}_\infty'$
there exists $F \in {\mathcal S}'$
such that $F|{\mathcal S}_\infty=f$.
\item
The linear spaces
${\mathcal S}'/{\mathcal P}$
and
${\mathcal S}_\infty$ are isomorphic.
\item
The mapping
$R:F \in {\mathcal S}' \mapsto F|{\mathcal S}_\infty
\in {\mathcal S}_\infty'$
is continuous.
\item
The mapping
$R:F \in {\mathcal S}' \mapsto F|{\mathcal S}_\infty
\in {\mathcal S}_\infty'$
is open,
that is,
$R(U)$ is an open set in ${\mathcal S}_\infty'$
for all open sets $U \subset {\mathcal S}'$.
\end{enumerate}
\end{theorem}

\subsubsection{\bf Remarks on Theorem \ref{thm:151209-1} and its proof}

Here we collect some facts for the proof of Theorem \ref{thm:151209-1}.
\begin{remark}\
\begin{enumerate}
\item
A direct consequence of this theorem is that
\[
[f] \in {\mathcal S}'/{\mathcal P}
\mapsto
f|{\mathcal S}_\infty \in {\mathcal S}'_\infty
\]
is a topological isomorphism.
\item
The continuity of $R$ is clear
from the definition.
\item
Since
${\mathcal S}'$ is NOT metrizable,
one can not apply the Baire category theorem.
It seems that there is no literature
that allows us to apply a version
of the Baire category theorem.
\item
It seems that one can not find
the proof of the openness of $R$
in any literature
before \cite{NNS15}.
The proof of $(1)$ and $(2)$ is known
in the book \cite[5.1.]{Triebel-text-83}.
Yuan, Sickel and Yang proved $(3)$ and $(4)$
 \cite[Proposition 8.1]{YSY10-2}.
However, there was a gap.
They used the closed graph theorem.
However, it seems unclear that
one can use the closed graph theorem.
The proof in this note is essentially
to close the gap of $(4)$
in the proof of \cite[Proposition 8.1]{YSY10-2}.
\item
The idea for the proof is 
to extend many functionals we deal with
to 
${\mathcal V}_N$ or
${\mathcal W}_N$ carefully,
where
\begin{align}
\label{eq:151224-13}
{\mathcal V}_N
&\equiv 
\overline{{\mathcal S}_\infty}{}^{p_N}
\subset
C^N\\
\label{eq:151224-14}
{\mathcal W}_N
&\equiv 
\overline{{\mathcal S}}{}^{p_N}
\subset
C^N.
\end{align}
\end{enumerate}
\end{remark}

\subsubsection{\bf $\ker R$}

We specify $\ker R$ here.
The following lemma is fundamental.
\begin{theorem}\label{thm:151211-1}
Let $f \in {\mathcal S}'$.
Then the following are equivalent.
\begin{enumerate}
\item
${\rm supp}(f) \subset \{0\}$.
\item
$f$ is expressed as the following finite sum:
$$\displaystyle
f=\sum_{\lambda \in \Lambda}
c_\lambda \partial^\lambda \delta_0,
$$
where
$\Lambda \subset {\mathbb N}_0{}^n$
is a finite set
and
$\partial^\lambda \delta_0$
denotes the distribution defined by
$\langle \partial^\lambda \delta_0,\varphi \rangle
 \equiv
(-1)^{|\lambda|}\partial^\lambda \varphi(0)$
for $\varphi \in {\mathcal S}$.
\end{enumerate}
\end{theorem}

Via the Fourier transform,
we can prove:
\begin{corollary}
$\ker(R)={\mathcal P}$.
\end{corollary}

\begin{proof}[{\bf Proof of Theorem \ref{thm:151211-1}}]
\
\begin{enumerate}
\item[$\circ$]
Suppose
${\rm supp}(f) \subset \{0\}$.
By Lemma \ref{lem:151211-1},
we can find $N \in {\mathbb N}$ such that
\begin{equation}\label{eq:151224-12}
|\langle f,\varphi \rangle|
\le
N\sum_{|\alpha| \le N}
\sup_{x \in {\mathbb R}^n}
\langle x \rangle^N|\partial^\alpha \varphi(x)|
\end{equation}
for all $\varphi \in {\mathcal S}$.
Choose 
$\{c_\alpha\}_{\alpha \in {\mathbb N}_0{}^n, |\alpha| \le 2N+1}$
such that
\begin{equation}\label{eq:151224-17}
\psi(x)=\varphi(x)-\sum_{|\alpha| \le 2N+1}
c_\alpha x^\alpha e^{-|x|^2}=O(|x|^{2N+2})
\end{equation}
as $x \to 0$.
We claim
\begin{equation}\label{eq:151224-11}
\langle f,\psi \rangle=0.
\end{equation}
Since each $c_\alpha$ depends 
continuously on $\varphi$, more precisely,
\[
c_\alpha=\sum_{\lambda \in \Lambda_\alpha}
k_{\alpha,\lambda}\partial^\lambda \varphi(0)
\]
with $\Lambda_\alpha$ a finite set in ${\mathbb N}_0{}^n$,
we have
\[
\langle f,\varphi \rangle
=
\sum_{|\alpha| \le 2N+1}
\langle f,x^\alpha e^{-|x|^2} \rangle c_\alpha.
\]
\item[$\circ$]
Suppose $f$ is expressed as the following finite sum
as in $(2)$.
Then we have
\[
\langle \partial^\alpha \delta_0,\varphi \rangle
=(-1)^{|\alpha|}
\partial^\alpha \varphi(0)=0.
\]
Thus ${\rm supp}(f) \subset \{0\}$ once we show
(\ref{eq:151224-11}).

Let us show (\ref{eq:151224-11}).
Recall that ${\mathcal V}_N$ is defined by
(\ref{eq:151224-13}).
Let $\tau:{\mathbb R} \to {\mathbb R}$ be a function
such that
$\chi_{(2,\infty)} \le \tau \le \chi_{(1,\infty)}$.
Then 
\[
\lim_{j \to \infty}
p_N(\psi-\tau(j \cdot)\psi)=0
\]
by the Leibniz rule,
(\ref{eq:151224-12}) and (\ref{eq:151224-17}).
Thus,
\[
\langle f,\psi \rangle
=
\lim_{j \to \infty}
\langle f,\tau(j \cdot)\psi \rangle
=0
\]
since $f$ is supported away from the origin.
This proves (\ref{eq:151224-11}).
\end{enumerate}
\end{proof}

\subsubsection{\bf Surjectivity of $R$}

We aim here to show that $R$ is surjective.
To this end we choose
$f \in {\mathcal S}_\infty'$
arbitrarily.
Then similar to Lemma \ref{lem:151211-1},
we can find
$N \in {\mathbb N}$
such that
\begin{equation}\label{eq:151209-1}
|\langle f,\varphi \rangle| \le N p_N(\varphi)
\end{equation}
for all $\varphi \in {\mathcal S}_\infty$.

Thanks to (\ref{eq:151209-1})
$f$ extends uniquely to a countinous linear functional
$g$ on the normed space ${\mathcal V}_N$,
i.e. $f=g|{\mathcal S}_\infty$.
We use the Hahn-Banach theorem
to extend $g$ to ${\mathcal W}_N$ to have
a continuous linear functional $G$
defined on ${\mathcal W}_N$ such that
$G|{\mathcal V}_N=g$ and that
$|\langle G,\Phi \rangle| \le N p_N(\Phi)$
for all $\Phi \in {\mathcal V}_N$.

\subsubsection{\bf Openness of $R$}

Everything hinges on the following observation:
\begin{theorem}\label{thm:151209-9}
Let $\varphi_1,\varphi_2,\ldots,\varphi_N \in {\mathcal S}_\infty$
and
let $\psi_1,\psi_2,\ldots,\psi_K \in {\mathcal S}$.
Assume that
the system
$\{[\psi_1],[\psi_2],\ldots,[\psi_K]\}$
is linearly independent 
in ${\mathcal S}/{\mathcal S}_\infty$.
Let 
\[
U
\equiv
\bigcap_{j=1}^N\{f \in {\mathcal S}_\infty'\,:\,
|\langle f,\varphi_j \rangle|<1\}
\]
be a neighborhood of $0=0_{{\mathcal S}_\infty'}
\in {\mathcal S}_\infty'$
and
\[
V
\equiv
\bigcap_{j=1}^N\{f \in {\mathcal S}'\,:\,
|\langle f,\varphi_j \rangle|<1\}
\cap
\bigcap_{k=1}^K\{f \in {\mathcal S}'\,:\,
|\langle f,\psi_k \rangle|<1\}
\]
a neighborhood of $0=0_{{\mathcal S}'}
\in {\mathcal S}'$.
Then $R(V)=U$.
\end{theorem}

To prove Theorem \ref{thm:151209-9},
we need the following key lemma:
\begin{lemma}\label{lem:151209-1}
Under the assumptions of
Theorem $\ref{thm:151209-9}$,
there exists $N \gg 1$ such that
the system
$\{[\psi_1],[\psi_2],\ldots,[\psi_K]\}$
is linearly independent 
in ${\mathcal W}_N/{\mathcal V}_N$.
\end{lemma}

\begin{proof}
Write
\[
S^{2K-1}\equiv 
\left\{
(a_1,a_2,\ldots,a_K) \in {\mathbb C}^K
\,:\,
\sum_{k=1}^K|a_k|^2=1
\right\}.
\]
We have only to show that
\[
\left\{
\sum_{k=1}^K a_k \psi_k\,:\,
(a_1,a_2,\ldots,a_K) \ne (0,0,\ldots,0) \in {\mathbb C}^K
\right\}
\cap {\mathcal V}_N
=\emptyset
\]
or equivalently
\[
\left\{
\sum_{k=1}^K a_k \psi_k\,:\,
(a_1,a_2,\ldots,a_K) \in S^{2K-1}
\right\}
\cap {\mathcal V}_N
=\emptyset
\]
for some large $N \gg n+1$.

For each
$(a_1,a_2,\ldots,a_K) \in S^{2K-1}$,
there exists
$\alpha(a_1,a_2,\ldots,a_K)
\in {\mathbb N}_0{}^n$
such that
\[
\int_{{\mathbb R}^n}x^{\alpha(a_1,a_2,\ldots,a_K)}
\sum_{j=1}^K a_j\psi_j(x)\,dx \ne 0.
\]
Since the function
$$\displaystyle
(b_1,b_2,\ldots,b_K) \in {\mathbb C}^K
\mapsto
\int_{{\mathbb R}^n}x^{\alpha(a_1,a_2,\ldots,a_K)}
\sum_{j=1}^K b_j\psi_j(x)\,dx \in {\mathbb C}
$$
is continuous,
there exists an open neighborhood
$U(a_1,a_2,\ldots,a_K) \subset {\mathbb C}^K$
of the point
$(a_1,a_2,\ldots,a_K) \in S^{2K-1}$ 
such that
\[
\int_{{\mathbb R}^n}x^{\alpha(a_1,a_2,\ldots,a_K)}
\sum_{j=1}^K b_j\psi_j(x)\,dx \ne 0
\]
for all 
$(b_1,b_2,\ldots,b_K) \in U(a_1,a_2,\ldots,a_K)$.
Since $S^{2K-1}$ is compact,
we can find a finite collection
$(a_1^l,a_2^l,\ldots,a_K^l)$
for $l=1,2,\ldots,L$
such that
\[
S^{2K-1} \subset \bigcup_{l=1}^L
U(a_1^l,a_2^l,\ldots,a_K^l).
\]
Taking
\[
N \equiv n+2+\sum_{l=1}^L |\alpha(a_1^l,a_2^l,\ldots,a_K^l)|,
\]
we obtain the desired result.
\end{proof}

We prove Theorem \ref{thm:151209-9}.
In fact, from Lemma \ref{lem:151209-1},
for any $f \in U$,
we can find $G \in {\mathcal W}_N$
so that
$G|{\mathcal V}_N=g$
and that
$\langle G,\psi_k \rangle=0$
for all $k=1,2,\ldots,K$.
In fact, we have a strictly increasing sequence
of closed linear subspaces
$\{{\rm Span}
({\mathcal V}_N \cup \{\psi_k\}_{k=1}^{K'})\}_{K'=1}^K$.
Apply successively the Hahn Banach theorem
to obtain the desired $G$.
Then observe that $G|{\mathcal S} \in R$ and $R(G|{\mathcal S})=f$.
Thus $F=G|{\mathcal S}$ does the job.

\begin{remark}\label{rem:151224-1}
Under the assumption of Theorem \ref{thm:151209-9},
set
\[
V_0
\equiv
\bigcap_{j=1}^N\{f \in {\mathcal S}'\,:\,
|\langle f,\varphi_j \rangle|<1\}
\cap
\bigcap_{k=1}^K\{f \in {\mathcal S}'\,:\,
\langle f,\psi_k \rangle=0\}.
\]
Then $R(V_0)=U$ as the above proof shows.
\end{remark}

\begin{theorem}\label{thm:151209-2}
Let $\varphi_1,\varphi_2,\ldots,\varphi_N \in {\mathcal S}_\infty$
and
let $\psi_1,\psi_2,\ldots,\psi_{\tilde{K}} \in {\mathcal S}$.
Assume that
the system
$\{[\psi_1],[\psi_2],\ldots,[\psi_K]\}$
is a maximal family of linearly independent elements
in
$\{[\psi_1],[\psi_2],\ldots,[\psi_{\tilde{K}}]\}$
in ${\mathcal S}/{\mathcal S}_\infty$,
so that $\tilde{K} \ge K$.
More precisely,
\begin{equation}\label{eq:151224-18}
\psi_k=\tilde{\varphi}_k+\sum_{l=1}^K a_{k l}\psi_l
\end{equation}
for some $\tilde{\varphi}_k \in {\mathcal S}_\infty$
and $a_{k l} \in {\mathbb C}$,
$k=\tilde{K}+1,\tilde{K}+2,\ldots,K$
and
$l=1,2,\ldots,K$.
Let 
\[
U \equiv \bigcap_{j=1}^N\{f \in {\mathcal S}_\infty'\,:\,
|\langle f,\varphi_j \rangle|<1\}
\cap
\bigcap_{k=1}^K\{f \in {\mathcal S}_\infty'\,:\,
|\langle f,\tilde{\varphi}_k \rangle|<1\}
\]
be a neighborhood of $0=0_{{\mathcal S}_\infty'}
\in {\mathcal S}_\infty'$
and
\[
V \equiv \bigcap_{j=1}^N\{f \in {\mathcal S}'\,:\,
|\langle f,\varphi_j \rangle|<1\}
\cap
\bigcap_{k=1}^K\{f \in {\mathcal S}'\,:\,
|\langle f,\psi_k \rangle|<1\}
\]
a neighborhood of $0=0_{{\mathcal S}'}
\in {\mathcal S}'$.
Then $R(U) \supset V$.
\end{theorem}

\begin{proof}
Let $f \in V$.
Then we can find a linear functional
$F \in {\mathcal S}'$
such that
$F|{\mathcal S}_\infty=f$
and that
\begin{equation}\label{eq:151224-19}
\langle F,\psi_k \rangle=0
\end{equation}
for $k=1,2,\ldots,\tilde{K}$ according to Remark \ref{rem:151224-1}.
Observe that
$$
\langle F,\psi_k \rangle
=\langle F,\tilde{\varphi}_k \rangle
\in \{z \in {\mathbb C}\,:\,|z|<1\}
$$
for $k=\tilde{K}+1,\tilde{K}+2,\ldots,K$
from (\ref{eq:151224-18}) and (\ref{eq:151224-19}).
Thus $R(U) \supset V$.
\end{proof}

\section{Homogeneous Besov spaces}

In this section we consider homogeneous Besov spaces.
As we will see, justifying the definition is one of the hard tasks.

\subsection{Definition}

We define homogeneous Besov spaces
and justify their definition.

\begin{definition}
Let $1 \le p,q \le \infty$ and $s \in {\mathbb R}$.
Choose $\varphi \in {\mathcal S}$
so that
$\chi_{B(4) \setminus B(2)} \le \varphi
\le \chi_{B(8) \setminus B(1)}$.
Define
\[
\dot{B}^s_{pq}
\equiv
\left\{
f \in {\mathcal S}'/{\mathcal P}\,:\,
\|f\|_{\dot{B}^s_{pq}}
\equiv
\left(
\sum_{j=-\infty}^\infty 
(2^{j s}\|{\mathcal F}^{-1}[\varphi_j \cdot {\mathcal F}f]\|_p)^q
\right)^{\frac{1}{q}}<\infty
\right\}.
\]
The space
$\dot{B}^s_{pq}$ is the set of all
$f \in {\mathcal S}'/{\mathcal P}$
for which the quasi-norm
$\|f\|_{\dot{B}^s_{pq}}$
is finite.
\end{definition}

A couple of helpful remarks may be in order.
\begin{remark}
\
\begin{enumerate}
\item
Since
${\mathcal S}'/{\mathcal P}$
is a quotient linear space,
the expression
$f \in {\mathcal S}'/{\mathcal P}$
is not appropriate.
Instead,
one should have written
$[f]\in{\mathcal S}'/{\mathcal P}$,
where
$[f]$ denotes the class
in ${\mathcal S}'/{\mathcal P}$
to which $f$ belongs.
Nevertheless
we write
$f \in {\mathcal S}'/{\mathcal P}$
by habit.
\item
It does not make sense
to consider
$\|f\|_{p}$ for $f \in {\mathcal S}'$;
${\mathcal S}'$ is NOT included in $L^1_{\rm loc}$.
However, the distribution
${\mathcal F}^{-1}[\varphi_j \cdot {\mathcal F}f]$
is a $C^\infty$-function as is seen from the identity
\[
{\mathcal F}^{-1}[\varphi_j \cdot {\mathcal F}f](x)
=
\frac{1}{\sqrt{(2\pi)^n}}
\langle f,{\mathcal F}^{-1}\varphi_j(x-\cdot)\rangle
\quad (x \in {\mathbb R}^n).
\]
\item
Note that $f \in {\mathcal P}$
if and only if
${\mathcal F}^{-1}[\varphi_j \cdot {\mathcal F}f]=0$
for all $j \in {\mathbb Z}$.
In fact when $f \in {\mathcal P}$,
$\varphi_j \cdot {\mathcal F}f=0$
for all $j \in {\mathbb Z}$,
since $f \in 
{\rm Span}(\{\partial^\alpha \delta_0\}_{\alpha \in {\mathbb N}_0{}^n})$.
If ${\mathcal F}^{-1}[\varphi_j \cdot {\mathcal F}f]=0$
for all $j \in {\mathbb Z}$,
then
$\varphi_j \cdot {\mathcal F}f=0$.
Choose a test function
$\tau \in C^\infty_{\rm c}$.
Since $\tau$ is supported away from the origin
and $\tau$ vanishes outside of a bounded set,
one has
\[
\tau=\sum_{j=-\infty}^\infty
\frac{\tau \varphi_j}{\Phi} \cdot \varphi_j,
\]
where 
\begin{equation}\label{eq:151210-1}
\Phi \equiv \sum_{j=-\infty}^\infty \varphi_j{}^2.
\end{equation}
It counts that the function
$\varphi_j/\Phi$ makes sense
as an element in $C^\infty_{\rm c}$;
compare the size of their support.
Also, the expression
(\ref{eq:151210-1}) is essentially a finite sum,
so that
\[
\langle \tau,{\mathcal F}f \rangle
=\sum_{j=-\infty}^\infty
\left<
\frac{\tau \varphi_j}{\Phi},\varphi_j \cdot {\mathcal F}f
\right>
=0.
\]
Thus, ${\rm supp}({\mathcal F}f) \subset \{0\}$,
implying that $f \in {\mathcal P}$.
\end{enumerate}
\end{remark}

The following observation justifies
the definition $\dot{B}^s_{pq}$ as a linear space.
In fact, we chose $\varphi$ 
so that the norm of 
$\dot{B}^s_{pq}$ depends on $\varphi$.
But as the following theorem shows,
$\dot{B}^s_{pq}$ is independent of $\varphi$
as a set, which justifies the definition of
$\dot{B}^s_{pq}$.
\begin{theorem}\label{thm:160118-3}
Let $1 \le p,q \le \infty$ and $s \in {\mathbb R}$.
Let $\varphi,\tilde{\varphi} \in {\mathcal S}$
satisfy
$\chi_{B(4) \setminus B(2)} \le \varphi,\tilde{\varphi}
\le \chi_{B(8) \setminus B(1)}$.
Set
$\varphi_j:\equiv\varphi(2^{-j}\cdot)$
and
$\tilde{\varphi}_j\equiv\tilde{\varphi}(2^{-j}\cdot)$
for $j \in {\mathbb Z}$.
Then we have
\begin{align*}
\left(
\sum_{j=-\infty}^\infty 
(2^{j s}\|{\mathcal F}^{-1}[\varphi_j \cdot {\mathcal F}f]\|_p)^q
\right)^{\frac{1}{q}}
\sim
\left(
\sum_{j=-\infty}^\infty 
(2^{j s}\|{\mathcal F}^{-1}[\tilde{\varphi}_j \cdot {\mathcal F}f]\|_p)^q
\right)^{\frac{1}{q}}
\end{align*}
for all $f \in {\mathcal S}'$.
\end{theorem}

\begin{proof}
By symmetry, it suffices to show that
\begin{align}\label{eq:151210-2}
\left(
\sum_{j=-\infty}^\infty 
(2^{j s}\|{\mathcal F}^{-1}[\varphi_j \cdot {\mathcal F}f]\|_p)^q
\right)^{\frac{1}{q}}
\lesssim
\left(
\sum_{j=-\infty}^\infty 
(2^{j s}\|{\mathcal F}^{-1}[\tilde{\varphi}_j \cdot {\mathcal F}f]\|_p)^q
\right)^{\frac{1}{q}}.
\end{align}
Define $\psi \in C^\infty_{\rm c}$ by
\[
\psi
\equiv
\frac{\varphi}{\tilde{\varphi}_{-1}+\tilde{\varphi}+\tilde{\varphi}_1}.
\]
Then we have
$\varphi_j
=\psi_j(\tilde{\varphi}_{j-1}+\tilde{\varphi}_j+\tilde{\varphi}_{j+1})$
for each $j \in {\mathbb Z}$.
Inserting this relation into the right-hand side of (\ref{eq:151210-2}),
we obtain
\begin{align*}
\lefteqn{
\left(
\sum_{j=-\infty}^\infty 
(2^{j s}\|{\mathcal F}^{-1}[\varphi_j \cdot {\mathcal F}f]\|_p)^q
\right)^{\frac{1}{q}}
}\\
&=
\left(
\sum_{j=-\infty}^\infty 
(2^{j s}\|{\mathcal F}^{-1}[
\psi_j\cdot
(\tilde{\varphi}_{j-1}+\tilde{\varphi}_j+\tilde{\varphi}_{j+1})
\cdot
{\mathcal F}f]\|_p)^q
\right)^{\frac{1}{q}}.
\end{align*}
By the relation between
the Fourier transform,
the convolution
and the pointwise multiplication,
we have
\begin{align*}
\lefteqn{
\left(
\sum_{j=-\infty}^\infty 
(2^{j s}\|{\mathcal F}^{-1}[\varphi_j \cdot {\mathcal F}f]\|_p)^q
\right)^{\frac{1}{q}}
}\\
&\simeq
\left(
\sum_{j=-\infty}^\infty 
(2^{j s}\|{\mathcal F}^{-1}\psi_j*
{\mathcal F}^{-1}[
(\tilde{\varphi}_{j-1}+\tilde{\varphi}_j+\tilde{\varphi}_{j+1})
\cdot
{\mathcal F}f]\|_p)^q
\right)^{\frac{1}{q}}.
\end{align*}
By the Young inequality and the fact that
$\|{\mathcal F}\varphi_j\|_1=\|{\mathcal F}\varphi_1\|_1$
for all $j \in {\mathbb N}$,
we obtain
\begin{align*}
\lefteqn{
\left(
\sum_{j=-\infty}^\infty 
(2^{j s}\|{\mathcal F}^{-1}[\varphi_j \cdot {\mathcal F}f]\|_p)^q
\right)^{\frac{1}{q}}
}\\
&\lesssim
\left(
\sum_{j=-\infty}^\infty 
(2^{j s}\|{\mathcal F}^{-1}\psi_j\|_1
\|{\mathcal F}^{-1}[
(\tilde{\varphi}_{j-1}+\tilde{\varphi}_j+\tilde{\varphi}_{j+1})
\cdot
{\mathcal F}f]\|_p)^q
\right)^{\frac{1}{q}}\\
&\simeq
\left(
\sum_{j=-\infty}^\infty 
(2^{j s}
\|{\mathcal F}^{-1}[
(\tilde{\varphi}_{j-1}+\tilde{\varphi}_j+\tilde{\varphi}_{j+1})
\cdot
{\mathcal F}f]\|_p)^q
\right)^{\frac{1}{q}}.
\end{align*}
Finally, by the triangle inequality and the index shiftings,
we have
\begin{align*}
\lefteqn{
\left(
\sum_{j=-\infty}^\infty 
(2^{j s}\|{\mathcal F}^{-1}[\varphi_j \cdot {\mathcal F}f]\|_p)^q
\right)^{\frac{1}{q}}
}\\
&\lesssim
\left(
\sum_{j=-\infty}^\infty 
(2^{j s}
\|{\mathcal F}^{-1}[\tilde{\varphi}_{j-1}{\mathcal F}f]\|_p)^q
\right)^{\frac{1}{q}}
+
\left(
\sum_{j=-\infty}^\infty 
(2^{j s}
\|{\mathcal F}^{-1}[\tilde{\varphi}_j \cdot {\mathcal F}f]\|_p)^q
\right)^{\frac{1}{q}}\\
&\quad+
\left(
\sum_{j=-\infty}^\infty 
(2^{j s}
\|{\mathcal F}^{-1}[\tilde{\varphi}_{j+1}{\mathcal F}f]\|_p)^q
\right)^{\frac{1}{q}}\\
&\simeq
\left(
\sum_{j=-\infty}^\infty 
(2^{j s}\|{\mathcal F}^{-1}[\tilde{\varphi}_j \cdot {\mathcal F}f]\|_p)^q
\right)^{\frac{1}{q}},
\end{align*}
which proves (\ref{eq:151210-2}).
\end{proof}

Before we conclude this section,
we have a remark helpful till the end of this section.
\begin{remark}\label{rem:151224-2}
Let $\zeta \in C^\infty_{\rm c}({\mathbb R}^n \setminus \{0\})$.
The above proof shows that 
we have many possibilities of $\zeta$
in the definition of $\dot{B}^s_{pq}$;
\[
\|f\|_{\dot{B}^s_{pq}}
\equiv
\left(
\sum_{j=-\infty}^\infty 
(2^{j s}\|{\mathcal F}^{-1}[\zeta_j \cdot {\mathcal F}f]\|_p)^q
\right)^{\frac{1}{q}}.
\]
A particular choise of $\zeta$;
$\zeta=\psi-\psi(2\cdot)$,
where $\chi_{B(1)} \le \psi \le \chi_{B(2)}$,
is useful for later considerations.
\end{remark}

\subsection{H\"{o}lder-Zygmund spaces and Besov spaces}

So far, we justified the definition 
of the homogeneous Besov spaces $\dot{B}^s_{pq}$
with $1 \le p,q \le \infty$ and $s \in {\mathbb R}$
as a linear space (or a normed space).
However, it is rather hard to show that
Besov space $\dot{B}^s_{pq}$ is complete.
As is often the case,
it is a burden to construct the limit point
when we are given a Cauchy sequence
in metric spaces.

The main aim of this section is twofold;
one is to create a tool to obtain a candidate of the limit
point when we are given a Cauchy sequence in the space $\dot{B}^s_{pq}$.
The other is to exhibit an example
showing that the homogeneous Besov space 
$\dot{B}^s_{\infty\infty}$
is useful.
In fact, as a special case
the homogeneous H\"{o}lder-Zygmund spaces
are realized with the scale $\dot{B}$.

\subsubsection{\bf The difference operator}

To define homogeneous H\"{o}lder-Zygmund spaces,
we need the notion of difference of higher order.
We start with the following elementary identity:
\begin{lemma}\label{lem:151224-1}
$\displaystyle
\sum_{l=1}^m (-1)^l l^m{}_m {\rm C}_l
=(-1)^m m!
$
for all $m \in {\mathbb N}$.
\end{lemma}

\begin{proof}
Compare the coefficient of $t^m$ of the function
\[
t^m+\cdots
=
(e^t-1)^m
=\sum_{l=0}^m e^{l t}{}_m {\rm C}_l(-1)^{m-l}.
\]
Then we have
\[
1=\sum_{l=1}^m (-1)^{l+m}\frac{l^m}{m!}{}_m {\rm C}_l,
\]
as was to be shown.
\end{proof}

Let $y \in {\mathbb R}^n$ 
and $f:{\mathbb R}^n \to {\mathbb R}$ be a mapping.
Define inductively $\Delta^m_y f$ by
\[
\Delta^1_y f \equiv f(\cdot+y)-f, \quad
\Delta^{m+1}_y f \equiv \Delta^m_y(\Delta^1_m f).
\]

Based on Lemma \ref{lem:151224-1},
we shall obtain a formula to connect the difference
with the operator
$f \in {\mathcal S}' \mapsto {\mathcal F}^{-1}[\varphi_j \cdot {\mathcal F}f]$.
Set
\[
\Phi
\equiv
\sum_{r=1}^m \sum_{l=1}^m
\frac{r^m}{(r l)^n}{}_m {\rm C}_r\cdot {}_m {\rm C}_l
(-1)^{r+l+m+1}\Psi\left(\frac{\cdot}{r l}\right)
\]
\begin{lemma}\label{lem:151224-101}
Let $f \in L^1_{\rm loc}$
and $\Psi \in C^\infty_{\rm c}$.
Then
\[
2^{j n}\Phi(2^j \cdot)*f(x)-m!\int_{{\mathbb R}^n}\Psi(y)\,dy
\times f(x)
=
\sum_{r=1}^m
(-1)^{r+m+1} r^m
\int_{{\mathbb R}^n}\Psi(y)\Delta^m_{2^{-j}r y}f(x)\,dy.
\]
\end{lemma}

\begin{proof}
First we observe that
\begin{align*}
\int_{{\mathbb R}^n}\Phi(y)\,dy
&=
\sum_{r=1}^m
\sum_{l=1}^m
r^m(-1)^{r+l+1+m}
{}_m {\rm C}_r
\cdot
{}_m {\rm C}_l
\int_{{\mathbb R}^n}\Psi(y)\,dy\\
&=(-1)^{m+1}
\sum_{r=1}^m r^m(-1)^r{}_m {\rm C}_r
\times
\sum_{l=1}^m (-1)^l{}_m {\rm C}_l
\times
\int_{{\mathbb R}^n}\Psi(y)\,dy\\
&=m!
\int_{{\mathbb R}^n}\Psi(y)\,dy
\end{align*}
using Lemma \ref{lem:151224-1}.
Thus, it follows that
\begin{align*}
\lefteqn{
2^{j n}\Phi(2^j \cdot)*f(x)-m!\int_{{\mathbb R}^n}\Psi(y)\,dy
\times f(x)
}\\
&=
\sum_{r=1}^m \sum_{l=1}^m
r^m{}_m {\rm C}_r\cdot {}_m {\rm C}_l
(-1)^{r+l+m+1}
\int_{{\mathbb R}^n}\Psi(y)f(x-2^{-j}r l y)\,dy\\
&\quad-m!
\int_{{\mathbb R}^n}\Psi(y)f(x)\,dy\\
&=
\sum_{r=1}^m \sum_{l=0}^m
r^m{}_m {\rm C}_r\cdot {}_m {\rm C}_l
(-1)^{r+l+m+1}
\int_{{\mathbb R}^n}\Psi(y)f(x-2^{-j}r l y)\,dy\\
&=
\sum_{r=1}^m 
r^m{}_m {\rm C}_r 
(-1)^{r+m+1}
\int_{{\mathbb R}^n}\Psi(y)\Delta^m_{2^{-j}r y}f(x-2^{-j}r l y)\,dy,
\end{align*}
as was to be shown.
\end{proof}

The equivalent expression 
in the next lemma is useful
when we consider the difference operator.

\begin{lemma}\label{lem:151224-102}
Suppose that $\Psi \in {\mathcal S}$ satisfies
\[
\chi_{B(4) \setminus B(2)}
\le
{\mathcal F}\Psi
\le
\chi_{B(8) \setminus B(1)}.
\]
Define $\Phi$ by:
\[
\Phi
\equiv
\sum_{r=1}^m \sum_{l=1}^m
\frac{r^m}{(r l)^n}{}_m {\rm C}_r\cdot {}_m {\rm C}_l
(-1)^{r+l+m+1}\Psi\left(\frac{\cdot}{r l}\right).
\]
\begin{enumerate}
\item
The function
${\mathcal F}\Phi$ is constant in a neighborhood
of the origin.
\item
Let $1 \le p,q \le \infty$ and $s \in {\mathbb R}$.
Then
by setting
\[
\varphi_j \equiv {\mathcal F}\Phi(2^{-j-1}\cdot)-
{\mathcal F}\Phi(2^{-j}\cdot) \quad (j \in {\mathbb Z})
\]
and
\[
\|f\|_{\dot{B}^s_{pq}}^\dagger
\equiv
\left(
\sum_{j=-\infty}^\infty 
(2^{j s}\|{\mathcal F}^{-1}[\varphi_j \cdot {\mathcal F}f]\|_p)^q
\right)^{\frac{1}{q}},
\]
we obtain a norm equivalent to
$\dot{B}^s_{p q}$.
\end{enumerate}
\end{lemma}

\begin{proof}
\
\begin{enumerate}
\item
Just observe
\[
{\mathcal F}\Phi
\equiv
\sum_{r=1}^m \sum_{l=1}^m
r^m{}_m {\rm C}_r\cdot {}_m {\rm C}_l
(-1)^{r+l+m+1}{\mathcal F}\Psi(r l \cdot).
\]
\item
This assertion follows from Remark \ref{rem:151224-2}.
\end{enumerate}
\end{proof}

\subsubsection{\bf The space ${\mathcal C}^s$}

Denote by ${\mathcal P}_r$
the set of all polynomials $f \in {\mathcal P}$
having degree at most $r$.
\begin{definition}
The {\it H\"older space} $\dot{\mathcal C}^s$
is the set of all continuous functions
$f$ on ${\mathbb R}^n$
for which the semi-norm
\[
\|f\|_{\dot{\mathcal C}^s}
\equiv 
\sup\left\{
\frac{\|\Delta^{[s+1]}_y f\|_{L^\infty}}{|y|^s}
\,:\,y \in {\mathbb R}^n
\right\}<\infty.
\]
Sometimes one considers
$\dot{\mathcal C}^s$ modulo
${\mathcal P}_{[s]}$.
\end{definition}

\begin{examples}
Denote by ${\mathbb C}^{{\mathbb R}^n}$
the set of all complex-valued functions
defined on ${\mathbb R}^n$
and by ${\rm C}$ the set of all continuous functions
defined on ${\mathbb R}^n$.
\begin{enumerate}
\item
Let $0<s<1$.
Then
\begin{align*}
\dot{\mathcal C}^s
&=
\{f \in {\rm C}\,:\,
|f(x+y)-f(x)| \lesssim|y|^s\mbox{ for all }x,y \in {\mathbb R}^n\}\\
&=
\{f \in {\mathbb C}^{{\mathbb R}^n}\,:\,
|f(x+y)-f(x)| \lesssim |y|^s\mbox{ for all }x,y \in {\mathbb R}^n\}
\end{align*}
because $|f(x+y)-f(x)| \lesssim |y|^s$
implies the continuity of $f$.
\item
Let $s=1$.
Then
\begin{align*}
\dot{\mathcal C}^s
&=
\{f \in {\rm C}\,:\,
|f(x+y)-2f(x)+f(x-y)| \lesssim |y|\mbox{ for all }x,y \in {\mathbb R}^n\}.
\end{align*}
Remark that this set is NOT equal to
\[
\dot{\mathcal C}^s
=
\{f \in {\mathbb C}^{{\mathbb R}^n}\,:\,
|f(x+y)-2f(x)+f(x-y)| \lesssim |y|\mbox{ for all }x,y \in {\mathbb R}^n\},
\]
as one can show using the {\it Hamel basis},
the basis of the linear space 
${\mathbb R}^n$ over ${\mathbb Q}$.
\cite[Proposition A.1]{MSSW10}.
\end{enumerate}
\end{examples}

\subsubsection{\bf The space $\dot{B}^s_{\infty\infty}$}

Let $s>0$ and $f \in \dot{B}^s_{\infty\infty}$.
Let us now choose $\psi \in {\mathcal S}$
so that
\begin{equation}\label{eq:151224-21}
\chi_{B(1)} \le \psi \le \chi_{B(2)}.
\end{equation}
Define
\begin{equation}\label{eq:151224-22}
\varphi_j \equiv \psi(2^{-j}\cdot)-\psi(2^{-j+1}\cdot)
\end{equation}
for $j \in {\mathbb Z}$.
Set
\begin{equation}\label{eq:151215-4}
H \equiv \sum_{j=1}^\infty
{\mathcal F}^{-1}[\varphi_j \cdot {\mathcal F}f].
\end{equation}

It is easy to see that $H$ is a continuous function
since
\[
\sum_{j=1}^\infty \|\varphi_j(D)f\|_\infty
\le 
\sum_{j=1}^\infty 2^{-j s}\|f\|_{\dot{B}^s_{\infty\infty}}
\sim\|f\|_{\dot{B}^s_{\infty\infty}}.
\]
The function $H$ is called the {\it high frequency part}
of $f$.

Meanwhile, the function
\[
\tilde{G} \equiv \sum_{j=1}^\infty
{\mathcal F}^{-1}[\varphi_j \cdot {\mathcal F}f].
\]
is called the {\it low frequency part}
of $f$.
The trouble is that 
there is no guarantee that the right-hand side
defining $\tilde{G}$ is convergent
in some suitable topology.
To circumbent this problem,
we consider its derivative as follows:
\begin{lemma}\label{lem:151224-3}
Let $f \in \dot{B}^s_{\infty\infty}$
with $s>0$.
Let $\psi \in {\mathcal S}$ satisfy
$(\ref{eq:151224-21})$.
Define
$\varphi_j$
by $(\ref{eq:151224-22})$ for $j \in {\mathbb Z}$.
If $\alpha \ge [s+1]$,
then
\begin{equation}\label{eq:151215-2}
G_\alpha \equiv
\sum_{j=-\infty}^0 
\partial^\alpha\left[
{\mathcal F}^{-1}[\varphi_j \cdot {\mathcal F}f]
\right]
\end{equation}
is convergent in $L^\infty$
and hence this function is smooth.
\end{lemma}

\begin{proof}
Note that
\begin{align*}
{\mathcal F}^{-1}[\varphi_j \cdot {\mathcal F}f]
&=
{\mathcal F}^{-1}[
(\varphi_{j-1}+\varphi_j+\varphi_{j+1})
\cdot
\varphi_j \cdot {\mathcal F}f]\\
&\simeq_n
{\mathcal F}^{-1}[\varphi_{j-1}+\varphi_j+\varphi_{j+1}]*
{\mathcal F}^{-1}[\varphi_j \cdot {\mathcal F}f]
\end{align*}
Write $\tau \equiv {\mathcal F}^{-1}\varphi$.
Then we have
\begin{align*}
{\mathcal F}^{-1}[\varphi_j \cdot {\mathcal F}f]
&\simeq 
(2^{(j-1)n}\tau(2^{j-1}\cdot)
+2^{j n}\tau(2^j\cdot)
+2^{(j+1)n}\tau(2^{j+1}\cdot))*
{\mathcal F}^{-1}[\varphi_j \cdot {\mathcal F}f].
\end{align*}
By the H\"{o}lder inequality,
we have
\begin{align*}
\lefteqn{
\|\partial^\alpha\left[
{\mathcal F}^{-1}[\varphi_j \cdot {\mathcal F}f]
\right]\|_\infty
}\\
&=
\|2^{(j-1)(|\alpha|+n)}\partial^\alpha\tau(2^{j-1}\cdot)
+2^{j(|\alpha|+n)}\partial^\alpha\tau(2^j\cdot)
+2^{(j+1)(|\alpha|+n)}\partial^\alpha\tau(2^{j+1}\cdot)\|_{1}\\
&\quad \times
\|{\mathcal F}^{-1}[\varphi_j \cdot {\mathcal F}f]\|_\infty\\
&\lesssim
2^{j|\alpha|}
\|{\mathcal F}^{-1}[\varphi_j \cdot {\mathcal F}f]\|_\infty.
\end{align*}
As a result
we have
\begin{align*}
\sum_{j=-\infty}^0 
\left\|\partial^\alpha\left[
{\mathcal F}^{-1}[\varphi_j \cdot {\mathcal F}f]\right]\right\|_\infty
&\lesssim
\sum_{j=-\infty}^0 
2^{j|\alpha|}
\|{\mathcal F}^{-1}[\varphi_j \cdot {\mathcal F}f]\|_\infty\\
&\le
\sum_{j=-\infty}^0 
2^{j|\alpha|-j s}
\|f\|_{\dot{B}^s_{\infty\infty}}\\
&=
\sum_{j=-\infty}^0 
2^{j([s+1]-s)}
\|f\|_{\dot{B}^s_{\infty\infty}}\\
&\simeq
\|f\|_{\dot{B}^s_{\infty\infty}}.
\end{align*}
\end{proof}

To construct a substitute of the lower part,
we need to depend on a geometric property
of ${\mathbb R}^n$;
$H^1_{\rm DR}({\mathbb R}^n)=0$.
Applying this fact, we can prove the following:
\begin{lemma}\label{lem:151215-1}
Let $N \in {\mathbb N}$.
Suppose that we have
$\{f_\alpha\}_{\alpha \in {\mathbb N}_0{}^n, |\alpha|=N}
\subset C^\infty$
satisfying
\[
\partial_\beta f_\alpha=\partial_{\beta'}f_{\alpha'}
\]
for all $\alpha,\alpha',\beta,\beta' \in {\mathbb N}_0{}^n$
with $|\alpha|=|\alpha'|=N$ and
$\alpha+\beta=\alpha'+\beta'$.
Then there exists $f \in C^\infty$
such that
$\partial^\alpha f=f_\alpha$.
\end{lemma}

The proof is in Appendix; see Section \ref{Section:A2}.
By using Lemma \ref{lem:151215-1},
we have the following control of the low frequency part:
\begin{corollary}\label{cor:160204-1}
Let $f \in \dot{B}^s_{\infty\infty}$
with $s>0$.
Let $\psi \in {\mathcal S}$ satisfy
$(\ref{eq:151224-21})$.
Define
$\varphi_j$
by $(\ref{eq:151224-22})$ for $j \in {\mathbb Z}$.
There exists a function $G \in C^\infty$
such that
\begin{equation}\label{eq:151215-3}
\partial^\alpha G=
\sum_{j=-\infty}^0 
\partial^\alpha
\left[
{\mathcal F}^{-1}[\varphi_j \cdot {\mathcal F}f]
\right]
\end{equation}
for all $\alpha \in {\mathbb N}_0$ with $|\alpha| \ge [s+1]$.
\end{corollary}

We further investigate
the property of $G$ in Corollary \ref{cor:160204-1}.
\begin{proposition}\label{prop:151224-2}
Let $s>0$ and $f \in \dot{B}^s_{\infty\infty}$.
Choose $\psi \in {\mathcal S}$
so that
$\chi_{B(1)} \le \psi \le \chi_{B(2)}$.
Define
$\varphi_j \equiv \psi(2^{-j}\cdot)-\psi(2^{-j+1}\cdot)$
for $j \in {\mathbb Z}$.
\begin{enumerate}
\item
Defien $G$ and $H$
by $(\ref{eq:151215-3})$ and $(\ref{eq:151215-4})$,
respectively.
Then
$f-(G+H) \in {\mathcal P}.$
\item
$G+H \in \dot{\mathcal C}^s$.
\item
If $G'$ is such that
\[
\partial^\alpha G'=
\sum_{j=-\infty}^0 
\partial^\alpha
\left[{\mathcal F}^{-1}[\varphi_j \cdot {\mathcal F}f]
\right]
\]
for all $\alpha \in {\mathbb N}_0$ with $|\alpha| \ge [s+1]$,
then $G'-G$ is a polynomial of order less than or equal to $[s]$.
\end{enumerate}
\end{proposition}

\begin{proof}
\
\begin{enumerate}
\item
Observe that
$\partial^\alpha\left[
{\mathcal F}^{-1}[\varphi_j \cdot {\mathcal F}(f-G-H)]\right]=0$
for all $\alpha \in {\mathbb N}_0$ with $|\alpha| \ge [s+1]$
and
for all $j \in {\mathbb Z}$.
Thus
${\mathcal F}^{-1}[\varphi_j \cdot {\mathcal F}(f-G-H)]=0$
for all $j \in {\mathbb Z}$.
Thus $f-G-H \in {\mathcal P}$.
\item
Write $F=G+H$.
We need to show that
\begin{equation}\label{eq:151215-9}
|\Delta^{[s+1]}_y F(x)| \lesssim |y|^s\|f\|_{\dot{B}^s_{\infty\infty}}.
\end{equation}
By the triangle inequality we have
\begin{equation}\label{eq:151215-7}
|\Delta^{[s+1]}_y{\mathcal F}^{-1}[\varphi_j \cdot {\mathcal F}f](x)| \lesssim 
\|{\mathcal F}^{-1}[\varphi_j \cdot {\mathcal F}f]\|_{\infty}.
\end{equation}
Meanwhile, by the mean value theorem 
we have
\begin{equation}\label{eq:151215-8}
|\Delta^{[s+1]}_y{\mathcal F}^{-1}
[\varphi_j \cdot {\mathcal F}f](x)| \lesssim 2^{j[s+1]}|y|^{[s+1]}
\|{\mathcal F}^{-1}[\varphi_j \cdot {\mathcal F}f]\|_{\infty}.
\end{equation}
Let $j_0 \in {\mathbb Z}$ be chosen so that
\begin{equation}\label{eq:160204-111}
1 \le 2^{j_0}|y|<2.
\end{equation}
We define
\[
H_{j_0} \equiv \sum_{j=j_0}^\infty 
{\mathcal F}^{-1}[\varphi_j \cdot {\mathcal F}f], \quad
G_{j_0} \equiv G+H-H_{j_0}.
\]
Once we prove
\begin{equation}\label{eq:151215-10}
|\Delta^{[s+1]}_y G_{j_0}(x)| \lesssim 
|y|^s\|f\|_{\dot{B}^s_{\infty\infty}}
\end{equation}
and
\begin{equation}\label{eq:151215-11}
|\Delta^{[s+1]}_y H_{j_0}(x)| \lesssim 
|y|^s\|f\|_{\dot{B}^s_{\infty\infty}},
\end{equation}
then we have
(\ref{eq:151215-9}).

To prove (\ref{eq:151215-10}),
we use the mean-value theorem 
and (\ref{eq:151215-8}) to have
\begin{align*}
|\Delta^{[s+1]}_y G_{j_0}(x)|
&\lesssim
|y|^{[s+1]}\|\nabla^{[s+1]}G_{j_0}\|_\infty\\
&=
|y|^{[s+1]}
\left\|
\sum_{j=-\infty}^{j_0-1}\nabla^{[s+1]}\left[
{\mathcal F}^{-1}[\varphi_j \cdot {\mathcal F}f]\right]
\right\|_\infty\\
&\le
\sum_{j=-\infty}^{j_0-1}
|y|^{[s+1]}
\left\|\nabla^{[s+1]}\left[
{\mathcal F}^{-1}[\varphi_j \cdot {\mathcal F}f]\right]
\right\|_\infty.
\end{align*}
Arguing as we did in Lemma \ref{lem:151224-3},
we have
\begin{align*}
|\Delta^{[s+1]}_y G_{j_0}(x)|
&\lesssim
\sum_{j=-\infty}^{j_0-1}
2^{j[s+1]}|y|^{[s+1]}
\|{\mathcal F}^{-1}[\varphi_j \cdot {\mathcal F}f]\|_{\infty}\\
&\le
\sum_{j=-\infty}^{j_0-1}
2^{j([s+1]-s)}|y|^{[s+1]}
\|f\|_{\dot{B}^s_{\infty\infty}}\\
&\sim
2^{j_0([s+1]-s)}|y|^{[s+1]}
\|f\|_{\dot{B}^s_{\infty\infty}}\\
&\sim
|y|^{s}
\|f\|_{\dot{B}^s_{\infty\infty}}
\end{align*}
using (\ref{eq:160204-111}).
To prove (\ref{eq:151215-11})
we use (\ref{eq:151215-7}) to have
\begin{align*}
|\Delta^{[s+1]}_y G_{j_0}(x)|
\le 2^{[s+1]}
\sum_{j=j_0}^\infty
\|{\mathcal F}^{-1}[\varphi_j \cdot {\mathcal F}f]\|_{\infty}
\lesssim
\sum_{j=j_0}^\infty
2^{-j s}\|f\|_{\dot{B}^s_{\infty\infty}}
\sim
|y|^{s}
\|f\|_{\dot{B}^s_{\infty\infty}},
\end{align*}
as was to be shown.
\item
As we did in $(1)$,
we ${\mathcal F}^{-1}[\varphi_j \cdot {\mathcal F}(G'-G)]=0$.
Thus,
$G-G'$ is a polynomial.
Since 
$\partial^\alpha G'=\partial^\alpha G$
for all $\alpha \in {\mathbb N}_0$ with $|\alpha| \ge [s+1]$,
$G'-G$ is a polynomial of order less than or equal to $[s]$.
\end{enumerate}
\end{proof}

\subsubsection{\bf The isomorphism ${\mathcal C}^s \sim \dot{B}^s_{\infty\infty}$}

We can summarize the above observations as follows:
\begin{theorem}\label{thm:151224-102}
Let $s>0$.
Choose an auxiliary $\psi \in {\mathcal S}$
so that $\chi_{B(1)} \le \psi \le \chi_{B(2)}$.
Set $\varphi\equiv\psi(2^{-1}\cdot)-\psi$.
\begin{enumerate}
\item
$\dot{\mathcal C}^s \hookrightarrow {\mathcal S}'_\infty$,
where
$\dot{\mathcal C}^s$ is the quotient linear space
modulo ${\mathcal P}_{[s]}$.
\item
For $F \in \dot{\mathcal C}^s$,
define
\[
f:\rho \in {\mathcal S}_\infty
\mapsto
\int_{{\mathbb R}^n}F(x)\rho(x)\,dx
\in {\mathbb C}.
\]
Then
$f \in \dot{B}^s_{\infty\infty}$.
\item
Let $f \in \dot{B}^s_{\infty\infty}$.
Then there exists a continuous function
$F \in \dot{\mathcal C}^s$ such that
$f-F \in {\mathcal P}$.
More precisely,
$F$ can be taken as a sum of continuous functions
$G$ and $H$ so that
\[
\partial^\alpha G=
\sum_{j=-\infty}^0
\partial^\alpha[\varphi_j(D)f], \quad
H=\sum_{j=1}^\infty \varphi_j(D)f.
\]
\end{enumerate}
\end{theorem}

\begin{proof}
\
\begin{enumerate}
\item
Suppose that
$f \in \dot{\mathcal C}^s$,
where we do NOT consider $\dot{\mathcal C}^s$
as the quotient space.
Write $m \equiv [s+1]$.
For $x \in {\mathbb R}$,
let us set
\[
g_k(y) \equiv 
\begin{cases}
\displaystyle
\sum_{l=0}^{m}(-1)^{m-l}
{}_{m}{\rm C}_l \cdot f(k y+l y)
&(k \ge 0),\\
\displaystyle
\sum_{l=-k}^{m}(-1)^{m-l}
{}_{m}{\rm C}_l \cdot f(k y+l y)
&(-m\le k < 0),\\
\displaystyle
0&
(k<-m)
\end{cases}
\]
for $y \in {\mathbb R}^n$.

Then we have
\[
f(k y)=
\sum_{l_m=-m}^k
\sum_{l_{m-1}=-m}^{l_m}
\cdots
\sum_{l_1=-m}^{l_2}g_k(y).
\]
For each $x$ with $|x|>1$,
choose $k \in {\mathbb N}$ so that
$k<|x| \le k+1$. Set $x=k y$.
Then we have
$|y| \ge k^{-1}|x|>1$ and 
\[
|f(x)| \lesssim (1+k)^m \le (1+|x|)^m
\]
as was to be shown.
\item
According to Lemma \ref{lem:151224-102}
and
the analogue
of Lemma \ref{lem:151224-101},
we have
\[
{\mathcal F}^{-1}\varphi_j*f(x)
=
\sum_{r=1}^m
(-1)^{r+m+1} r^m
\int_{{\mathbb R}^n}\Psi(y)
(\Delta^m_{2^{-j-1}r y}F(x)-\Delta^m_{2^{-j}r y}F(x))\,dy.
\]
Thus
\begin{align*}
|{\mathcal F}^{-1}\varphi_j*f(x)|
&\le
\sum_{r=1}^m
\int_{{\mathbb R}^n}|\Psi(y)|(
|\Delta^m_{2^{-j-1}r y}F(x)|+|\Delta^m_{2^{-j}r y}F(x)|)\,dy\\
&\le
\|F\|_{\dot{\mathcal C}^s}
\sum_{r=1}^m
\int_{{\mathbb R}^n}|\Psi(y)|\cdot|2^{-j}y|^s\,dy\\
&\simeq
2^{-j s}\|F\|_{\dot{\mathcal C}^s},
\end{align*}
and hence
$
2^{j s}|{\mathcal F}^{-1}\varphi_j*f(x)|
\lesssim
\|F\|_{\dot{\mathcal C}^s}
$
for $x \in {\mathbb R}^n$,
as was to be shown.
\item
This is included in Proposition \ref{prop:151224-2}.
\end{enumerate}
\end{proof}

\subsection{Fundamental properties}

Here we investigate fundamental properties
of homogeneous Besov spaces.
Let $\psi \in {\mathcal S}$ satisfy
$(\ref{eq:151224-21})$.
Define
$\varphi_j$
by $(\ref{eq:151224-22})$ for $j \in {\mathbb Z}$.
We set
\[
\|f\|_{\dot{B}^s_{pq}}
\equiv
\left(
\sum_{j=-\infty}^\infty 
(2^{j s}\|{\mathcal F}^{-1}[\varphi_j \cdot {\mathcal F}f]\|_p)^q
\right)^{\frac{1}{q}}
\]
for $f \in {\mathcal S}'/{\mathcal P}$
as before.

If $f \in {\mathcal S}_\infty$,
then the notation $[f] \in {\mathcal S}'/{\mathcal P}$
makes sense.
Here we can say more about this and we
can present some fundamental examples
of elements in $\dot{B}^s_{pq}$.
\begin{theorem}\label{thm:151210-7}
Let $1 \le p,q \le \infty$ and $s \in {\mathbb R}$.
Then
${\mathcal S}_\infty \hookrightarrow \dot{B}^s_{pq}$.
\end{theorem}

We use the following lemma:
\begin{lemma}\label{lem:151210-2}
Let $x \in {\mathbb R}^n$ and $j,k \in {\mathbb Z}$.
Then we have
\[
\int_{{\mathbb R}^n}
\frac{2^{(j+k)n}\,dy}{(4^j|y|^2+1)^{\frac{n+1}{2}}(4^k|x-y|^2+1)^{\frac{n+1}{2}}}
\sim
\frac{2^{\min(j,k)n}}{(4^{\min(j,k)}|x|^2+1)^{\frac{n+1}{2}}}.
\]
\end{lemma}

\begin{proof}
The proof is simple;
by the Fourier transform and the Plancherel theorem,
we have
\begin{align*}
\lefteqn{
\int_{{\mathbb R}^n}
\frac{2^{(j+k)n}\,dy}{(4^j|y|^2+1)^{\frac{n+1}{2}}(4^k|x-y|^2+1)^{\frac{n+1}{2}}}
}\\
&\simeq
\int_{{\mathbb R}^n}
\exp(-\xi \cdot x i-(2^{-j}+2^{-k})|\xi|)\,d\xi\\
&\simeq
\frac{(2^{-j}+2^{-k})^n}{((2^{-j}+2^{-k})^{-2}|x|^2+1)^{\frac{n+1}{2}}}
\sim
\frac{2^{\min(j,k)n}}{(4^{\min(j,k)}|x|^2+1)^{\frac{n+1}{2}}}.
\end{align*}
\end{proof}
We now prove Theorem \ref{thm:151210-7}.

\begin{proof}[Proof of Theorem \ref{thm:151210-7}]
Let $\theta \in {\mathcal S}_\infty$.
We seek to show
\[
\|\theta\|_{\dot{B}^s_{pq}}
=\left(
\sum_{j=-\infty}^\infty 
(2^{j s}\|{\mathcal F}^{-1}[\varphi_j \cdot {\mathcal F}\theta]\|_p)^q
\right)^{\frac{1}{q}}<\infty.
\]
Let $x \in {\mathbb R}^n$ be fixed.
It matters that
$\tilde{\theta} \equiv {\mathcal F}^{-1}(|\xi|^{-2L}{\mathcal F}\theta)$
belongs to ${\mathcal S}_\infty$
for all $L \in {\mathbb N}$
and that
$\Delta \tilde{\theta}=(-1)^L\theta$.
Then we have
\begin{align*}
{\mathcal F}^{-1}\varphi_j*\theta(x)
&=2^{j n}
\int_{{\mathbb R}^n}{\mathcal F}^{-1}\varphi(2^j y)\theta(x-y)\,dy\\
&=(-1)^L2^{j n}
\int_{{\mathbb R}^n}{\mathcal F}^{-1}\varphi(2^j y)\Delta^L\theta(x-y)\,dy\\
&=(-1)^L2^{j n+2j L}
\int_{{\mathbb R}^n}(\Delta^L{\mathcal F}^{-1}\varphi)(2^j y)\theta(x-y)\,dy.
\end{align*}
From Lemma \ref{lem:151210-2},
we obtain
\[
|{\mathcal F}^{-1}\varphi_j*\theta(x)|
\le \frac{2^{2j L}}{(|x|+1)^{n+1}}.
\]
Likewise
by letting
$\tilde{\varphi}(\xi) \equiv |\xi|^{-2L}\varphi(\xi)$,
we obtain
\[
|{\mathcal F}^{-1}\varphi_j*\theta(x)|
\lesssim \frac{2^{-2j L}}{(2^j|x|+1)^{n+1}}
\le \frac{2^{-(n+1)j-2j L}}{(|x|+1)^{n+1}}.
\]
Let $L \gg 1$.
As a result,
\[
\|\theta\|_{\dot{B}^s_{pq}}
\lesssim \left(
\sum_{j=-\infty}^\infty 
(2^{j s+|j|(n+1)-2|j|L})^q
\right)^{\frac{1}{q}}<\infty,
\]
as was to be shown.
\end{proof}

Next we aim to consider the role of the parameter $s$.
Let $f \in {\mathcal S}_\infty'$.
Then
$
f=\sum_{j=-\infty}^\infty 
{\mathcal F}^{-1}[\varphi_j \cdot {\mathcal F}f]
$
in ${\mathcal S}_\infty'$
since
$
\Phi=\sum_{j=-\infty}^\infty 
{\mathcal F}^{-1}[\varphi_j \cdot {\mathcal F}\Phi]
$
in ${\mathcal S}_\infty$
for all $\Phi \in {\mathcal S}_\infty$.
Observe that
\[
(-\Delta)^{\frac{\alpha}{2}}\Phi
\equiv
\sum_{j=-\infty}^\infty 
{\mathcal F}^{-1}[|\xi|^\alpha \cdot \varphi_j \cdot {\mathcal F}\Phi],
\]
where the convergence takes place in 
${\mathcal S}_\infty$.
Therefore,
we can define
$(-\Delta)^{\frac{\alpha}{2}}f$ by
\[
(-\Delta)^{\frac{\alpha}{2}}f
\equiv
\sum_{j=-\infty}^\infty 
{\mathcal F}^{-1}[|\xi|^\alpha \cdot \varphi_j \cdot {\mathcal F}f],
\]
where the convergence takes place in 
${\mathcal S}_\infty'$.
It is not so hard to show that
the definition
of 
$(-\Delta)^{\frac{\alpha}{2}}f$ 
does not depend on $\varphi$ satisfying
(\ref{eq:151224-2}).

\begin{theorem}\label{thm:151224-103}
Let $1 \le p,q \le \infty$ and $s,\alpha \in {\mathbb R}$.
Then
$(-\Delta)^{\frac{\alpha}{2}}:\dot{B}^s_{pq} \to \dot{B}^{s-\alpha}_{pq}$
is an isomorphism.
\end{theorem}

\begin{proof}
We have only to show that
$(-\Delta)^{\frac{\alpha}{2}}:\dot{B}^s_{pq} \to \dot{B}^{s-\alpha}_{pq}$
is continuous.
Indeed, once this is proved,
then we have
$(-\Delta)^{-\frac{\alpha}{2}}:\dot{B}^{s-\alpha}_{pq} \to \dot{B}^s_{pq}$
is continuous,
which is the inverse of
$(-\Delta)^{\frac{\alpha}{2}}$.

Let $f \in \dot{B}^s_{pq}$.
Then we have
\begin{align*}
\lefteqn{
\|(-\Delta)^{\frac{\alpha}{2}}f\|_{\dot{B}^{s-\alpha}_{pq}}
}\\
&=
\left(
\sum_{j=-\infty}^\infty 
(2^{j(s-\alpha)}\|
{\mathcal F}^{-1}[|\xi|^\alpha\cdot\varphi_j \cdot {\mathcal F}f]\|_p)^q
\right)^{\frac{1}{q}}\\
&=
\left(
\sum_{j=-\infty}^\infty 
(2^{j(s-\alpha)}\|{\mathcal F}^{-1}[(\varphi_{j-1}+\varphi_j+\varphi_{j+1})
\cdot|\xi|^\alpha \cdot \varphi_j \cdot {\mathcal F}f]\|_p)^q
\right)^{\frac{1}{q}}\\
&\simeq
\left(
\sum_{j=-\infty}^\infty 
(2^{j(s-\alpha)}\|{\mathcal F}^{-1}[(\varphi_{j-1}+\varphi_j+\varphi_{j+1})
\cdot|\xi|^\alpha]*{\mathcal F}^{-1}[\varphi_j \cdot {\mathcal F}f]\|_p)^q
\right)^{\frac{1}{q}}.
\end{align*}
If we use the Young inequality, then we have
\begin{align*}
\lefteqn{
\|(-\Delta)^{\frac{\alpha}{2}}f\|_{\dot{B}^{s-\alpha}_{pq}}
}\\
&\le
\left(
\sum_{j=-\infty}^\infty 
(2^{j(s-\alpha)}\|{\mathcal F}^{-1}[(\varphi_{j-1}+\varphi_j+\varphi_{j+1})
\cdot|\xi|^\alpha]\|_1
\|{\mathcal F}^{-1}[\varphi_j \cdot {\mathcal F}f]\|_p)^q
\right)^{\frac{1}{q}}\\
&\simeq
\left(
\sum_{j=-\infty}^\infty 
(2^{j s}\|{\mathcal F}^{-1}[\varphi_j \cdot {\mathcal F}f]\|_p)^q
\right)^{\frac{1}{q}}\\
&=
\|f\|_{\dot{B}^s_{pq}},
\end{align*}
as was to be shown.
\end{proof}

\begin{theorem}\label{thm:151224-101}
Let $1 \le p,q,r \le \infty$ and $s \in {\mathbb R}$.
\begin{enumerate}
\item
If $r \ge q$, then
$\dot{B}^s_{pq} \hookrightarrow \dot{B}^{s}_{pr}$.
\item
In the sense of continuous embedding
\begin{equation}\label{eq:160119-1}
\dot{B}^s_{pq} \hookrightarrow \dot{B}^{s-n/p}_{\infty q}.
\end{equation}
\end{enumerate}
\end{theorem}

\begin{proof}
\
\begin{enumerate}
\item
This is clear because
$\ell^q({\mathbb Z}) \hookrightarrow \ell^r({\mathbb Z})$.
\item
Let $f \in \dot{B}^s_{pq}$.
Then we have
\begin{align*}
\lefteqn{
\|f\|_{\dot{B}^{s-n/p}_{\infty q}}
}\\
&=
\left(
\sum_{j=-\infty}^\infty 
(2^{j(s-n/p)}
\|{\mathcal F}^{-1}[\varphi_j \cdot {\mathcal F}f]\|_\infty)^q
\right)^{\frac{1}{q}}\\
&=
\left(
\sum_{j=-\infty}^\infty 
(2^{j(s-n/p)}
\|{\mathcal F}^{-1}[
(\varphi_{j-1}+\varphi_j+\varphi_{j+1}) \cdot
\varphi_j \cdot {\mathcal F}f]\|_\infty)^q
\right)^{\frac{1}{q}}\\
&\simeq
\left(
\sum_{j=-\infty}^\infty 
(2^{j(s-n/p)}
\|{\mathcal F}^{-1}[
(\varphi_{j-1}+\varphi_j+\varphi_{j+1})]*{\mathcal F}^{-1}[
\varphi_j \cdot {\mathcal F}f]\|_\infty)^q
\right)^{\frac{1}{q}}.
\end{align*}
By the H\"{o}lder inequality, we have
\begin{align*}
\lefteqn{
\|f\|_{\dot{B}^{s-n/p}_{\infty q}}
}\\
&\lesssim
\left(
\sum_{j=-\infty}^\infty 
(2^{j(s-n/p)}
\|{\mathcal F}^{-1}[
(\varphi_{j-1}+\varphi_j+\varphi_{j+1})]\|_{p'}
\|{\mathcal F}^{-1}[
\varphi_j \cdot {\mathcal F}f]\|_p)^q
\right)^{\frac{1}{q}}\\
&\simeq
\left(
\sum_{j=-\infty}^\infty 
(2^{j s}\|{\mathcal F}^{-1}[\varphi_j \cdot {\mathcal F}f]\|_p)^q
\right)^{\frac{1}{q}}\\
&=
\|f\|_{\dot{B}^s_{pq}}.
\end{align*}
Thus we have the desired result.
\end{enumerate}
\end{proof}

\begin{theorem}\label{thm:160118-2}
Let $1 \le p,q \le \infty$ and $s \in {\mathbb R}$.
\begin{enumerate}
\item
$\dot{B}^s_{pq} \hookrightarrow {\mathcal S}_\infty'$.
\item
$\dot{B}^s_{pq}$ is complete.
\end{enumerate}
\end{theorem}

\begin{proof}\
\begin{enumerate}
\item
According to Theorems
\ref{thm:151224-102}
and
\ref{thm:151224-101},
\begin{equation}\label{eq:151224-104}
\dot{B}^s_{pq} 
\hookrightarrow
\dot{B}^{s-n/p}_{\infty q}
\hookrightarrow
\dot{B}^{s-n/p}_{\infty\infty}
\sim
\dot{\mathcal C}^{s-n/p}
\hookrightarrow {\mathcal S}_\infty'
\end{equation}
as long as $s>n/p$.
Thus,
when $s>n/p$,
$\dot{B}^s_{pq} \hookrightarrow {\mathcal S}_\infty'$.
According to Theorem \ref{thm:151224-103},
we still have
$\dot{B}^s_{pq} \hookrightarrow {\mathcal S}_\infty'$
when $s \le n/p$.
\item
According to Theorem \ref{thm:151224-103},
we may assume $s>n/p$.
Let $\{f_j\}_{j=1}^\infty$ be a Cauchy sequence
in
$\dot{B}^s_{pq}$.
Then, according to (\ref{eq:151224-104})
we have
$\dot{B}^s_{pq}\hookrightarrow\dot{\mathcal C}^{s-n/p}$.
Thus $\{f_j\}_{j=1}^\infty$ is a Cauchy sequence
in $\dot{\mathcal C}^{s-n/p}$.
Since $\dot{\mathcal C}^{s-n/p}$ is complete,
$\{f_j\}_{j=1}^\infty$ is convergent 
to $f \in \dot{\mathcal C}^{s-n/p}$.
Since
$\dot{\mathcal C}^{s-n/p} \hookrightarrow {\mathcal S}_\infty'$
and
${\mathcal F}^{-1}[\varphi_j \cdot {\mathcal F}f_l](x)
\simeq
\langle {\mathcal F}^{-1}\varphi_j(x-\cdot),f_l \rangle$,
we have
\[
\lim_{l \to \infty}
{\mathcal F}^{-1}[\varphi_j \cdot {\mathcal F}f_l](x)
=
{\mathcal F}^{-1}[\varphi_j \cdot {\mathcal F}f](x).
\]
By the Fatou theorem we have
\begin{align*}
\|f-f_k\|_{\dot{B}^s_{pq}}
&=
\left(
\sum_{j=-\infty}^\infty 
(2^{j s}\|{\mathcal F}^{-1}[\varphi_j \cdot {\mathcal F}[f-f_k]]\|_p)^q
\right)^{\frac{1}{q}}\\
&\le
\left(
\sum_{j=-\infty}^\infty \liminf_{l \to\infty}
(2^{j s}\|{\mathcal F}^{-1}[\varphi_j \cdot {\mathcal F}[f_l-f_k]]\|_p)^q
\right)^{\frac{1}{q}}\\
&\le\liminf_{l \to\infty}
\left(
\sum_{j=-\infty}^\infty 
(2^{j s}\|{\mathcal F}^{-1}[\varphi_j \cdot {\mathcal F}[f_l-f_k]]\|_p)^q
\right)^{\frac{1}{q}}.
\end{align*}
Thus, by letting $k \to \infty$,
we learn 
$f_k \to f$ in $\dot{B}^s_{p q}$.
\end{enumerate}
\end{proof}

\subsection{Realization}

When we consider the partial differential equation,
it is not confortable to work on the quotient space.
One of the reasons is that the quotient space does not give
us any information of the value of functions.
Therefore, at least we want to go back to
the subspace of ${\mathcal S}'$.
Although the evaluation mapping does not make
sense in ${\mathcal S}'$, we feel that
the situation becomes better in ${\mathcal S}'$
than in ${\mathcal S}_\infty' \simeq {\mathcal S}'/{\mathcal P}$.
Such a situation is available when $s$ is small enough.

\begin{theorem}\label{thm:151215-111}
Let $1 \le p,q \le \infty$.
Assume
\[
s<\frac{n}{p}
\]
or
\[
s=\frac{n}{p}, \quad q=1.
\]
Choose $\psi \in {\mathcal S}$
so that $\chi_{B(1)} \le \psi \le \chi_{B(2)}$.
Set $\varphi_j \equiv \psi(2^{-j}\cdot)-\psi(2^{-j+1}\cdot)$.
Then for all $f \in \dot{B}^s_{pq}$,
$\displaystyle
\sum_{j=-\infty}^0 
{\mathcal F}^{-1}[\varphi_j \cdot {\mathcal F}f]
$
is convergent in $L^\infty$
and
$\displaystyle
\sum_{j=1}^\infty 
{\mathcal F}^{-1}[\varphi_j \cdot {\mathcal F}f]
$
is convergent in ${\mathcal S}'$.
\end{theorem}

\begin{proof}
We use
$
\|{\mathcal F}^{-1}[\varphi_j \cdot {\mathcal F}f]\|_\infty
\lesssim
2^{j n/p}\|{\mathcal F}^{-1}[\varphi_j \cdot {\mathcal F}f]\|_p
$
to show that 
$\displaystyle
\sum_{j=-\infty}^0 
{\mathcal F}^{-1}[\varphi_j \cdot {\mathcal F}f]
$
is convergent in $L^\infty$
as before.
The fact that
$\displaystyle
\sum_{j=1}^\infty 
{\mathcal F}^{-1}[\varphi_j \cdot {\mathcal F}f]
$
is convergent in ${\mathcal S}'$
is a general fact.
Or to show this,
we can use the lift operator to have
$\displaystyle
\sum_{j=1}^\infty 
(-\Delta)^{-L}{\mathcal F}^{-1}[\varphi_j \cdot {\mathcal F}f]
$
is convergent in ${\mathcal S}'$
for any $L \in {\mathbb N}$.
In fact, we can check that
$$\displaystyle
\sum_{j=1}^\infty 
\|(-\Delta)^{-L}{\mathcal F}^{-1}[\varphi_j \cdot {\mathcal F}f]\|_\infty
\lesssim
\sum_{j=1}^\infty 
2^{j(n/p-2L)}\|{\mathcal F}^{-1}[\varphi_j \cdot {\mathcal F}f]\|_p
<\infty
$$
as long as $2L>\dfrac{n}{p}$.
\end{proof}

\subsection{Refinement of the Fourier transform}
\label{s3.5}

We supplement and prove Theorem \ref{thm:151224-1}.
\begin{theorem}\label{thm:151224-109}
\
\begin{enumerate}
\item
The Fourier transform
${\mathcal F}$
maps $L^1$ into $\dot{B}^0_{\infty 1}$.
\item
$\dot{B}^0_{\infty 1} \hookrightarrow {\rm BC}$,
where $\dot{B}^0_{\infty 1}$ can be regarded
as the subset of ${\mathcal S}'$
by way of Theorem \ref{thm:151215-111}.
\end{enumerate}
\end{theorem}

\begin{proof}
We assume that $\varphi \in {\mathcal S}$
satisfies
\[
{\rm supp}(\varphi) \subset B(4) \setminus B(1), \quad
\sum_{j=-\infty}^\infty \varphi_j
=\chi_{{\mathbb R}^n \setminus \{0\}},
\]
where $\varphi_j(\xi) \equiv \varphi(2^{-j}\xi)$
for $j \in {\mathbb Z}$ and $\xi \in {\mathbb R}^n$.
\begin{enumerate}
\item
$\displaystyle
\|{\mathcal F}f\|_{\dot{B}^0_{\infty 1}}
=
\sum_{j=-\infty}^\infty
\|{\mathcal F}^{-1}[\varphi_j\cdot f(-\cdot)]\|_\infty
\le
\sum_{j=-\infty}^\infty
\|\varphi_j\cdot f(-\cdot)\|_1
=
\|f\|_1
$
where we used the original Hausdorff-Young theorem
to obtain the inequality.
\item
The assumption
$f \in \dot{B}^0_{\infty 1}$ is equivalent to
\[
\sum_{j=-\infty}^\infty 
\|{\mathcal F}^{-1}[\varphi_j \cdot {\mathcal F}f]\|_\infty
<\infty.
\]
Thus,
the series
\[
f=\sum_{j=-\infty}^\infty 
{\mathcal F}^{-1}[\varphi_j \cdot {\mathcal F}f]
\]
converges uniformly.
Since
${\mathcal F}^{-1}[\varphi_j \cdot {\mathcal F}f]
\in {\rm BC}$,
it follows that
$f \in {\rm BC}$.
\end{enumerate}
\end{proof}

From the remark below we see that
Theorem \ref{thm:151224-109} improves
the classical Hausdorff-Young theorem.
\begin{remark}
Remark that
${\rm BUC}$ is a closed subspace of $L^\infty$,
which is of course equipped with the $L^\infty$-norm.

This implies that
${\rm BUC} \ne \dot{B}^0_{\infty 1}$.
In fact, the above proof shows that
${\rm BUC} \supset \dot{B}^0_{\infty 1}$.
Assume that equality holds.
Then the function
\[
f_j(x)\equiv \max(1,\min(-1,j x_1)) \in {\rm BUC}
\]
forms a bounded set.
This implies that
$\{f_j\}_{j=1}^\infty \subset \dot{B}^0_{\infty 1}$
is a bounded set by virtue of the closed graph theorem.
Since
$f_j \to 2\chi_{\{x_1>0\}}-1$
in ${\mathcal S}'$,
$2\chi_{\{x_1>0\}}-1$ would be a member
in $\dot{B}^0_{\infty 1}$.
This is a contradiction to Theorem \ref{thm:151224-109}(2).
\end{remark}

\section{More about function spaces}

\subsection{The homogeneous Besov space $\dot{B}^s_{pq}$ for $0<p,q \le \infty$ and $s \in {\mathbb R}$}
\label{section:4.1}

We extend $\dot{B}^s_{pq}$
for $0<p,q \le \infty$ and $s \in {\mathbb R}$.

\begin{definition}
Let $0<p,q \le \infty$ and $s \in {\mathbb R}$.
Choose $\varphi \in {\mathcal S}$
so that
$\chi_{B(4) \setminus B(2)} \le \varphi
\le \chi_{B(8) \setminus B(1)}$.
Define
\[
\dot{B}^s_{pq}
\equiv
\left\{
f \in {\mathcal S}'/{\mathcal P}\,:\,
\|f\|_{\dot{B}^s_{pq}}
\equiv
\left(
\sum_{j=-\infty}^\infty 
(2^{j s}\|{\mathcal F}^{-1}[\varphi_j \cdot {\mathcal F}f]\|_p)^q
\right)^{\frac{1}{q}}<\infty
\right\}.
\]
The space
$\dot{B}^s_{pq}$ is the set of all
$f \in {\mathcal S}'/{\mathcal P}$
for which the quasi-norm
$\|f\|_{\dot{B}^s_{pq}}$
is finite.
\end{definition}

Let $0<\eta<\infty$.
We define the powered Hardy-Littlewood maximal operator
$M^{(\eta)}$ by
\begin{equation}
M^{(\eta)}f(x)
\equiv \sup_{R>0}
\left(
\frac{1}{|B(x,R)|}\int_{B(x,R)}|f(y)|^\eta\,d y
\right)^\frac{1}{\eta}.
\end{equation}
The powered Hardy-Littlewood maximal operator 
$M^{(\eta)}$ 
comes about
in the context of the Plancherel-Polya-Nikolskii inequality;
\[
\sup_{y \in {\mathbb R}^n}
\frac{|f(x-y)|}{(1+R y)^{n/\eta}}
\lesssim 
M^{(\eta)}(x)
\]
for all $R>0$ and $f \in {\mathcal S}'$ such that
${\rm supp}(f) \subset B(R)$.

The thrust of considering
the spaces $\dot{B}^s_{pq}$ with
$0<p,q \le \infty$ and $s \in {\mathbb R}$
is the pointwise product
$f \cdot g$,
which corresponds to the fact that
$L^p \cdot L^q=L^r$
for any $p,q,r>0$ with $\frac1r=\frac1p+\frac1q$.

\subsection{The homogeneous Triebel-Lizorkin space $\dot{F}^s_{pq}$ for $0<p<\infty$, $0<q \le \infty$ and $s \in {\mathbb R}$}
\label{section:4.2}

We let 
$0<p<\infty$,
$0<q \le \infty$
and $s \in {\mathbb R}$
instead of letting
$0<p,q \le \infty$
and $s \in {\mathbb R}$.

\begin{definition}
Let $0<p,q \le \infty$ and $s \in {\mathbb R}$.
Choose $\varphi \in {\mathcal S}$
so that
$\chi_{B(4) \setminus B(2)} \le \varphi
\le \chi_{B(8) \setminus B(1)}$.
Define
\[
\dot{F}^s_{pq}
\equiv
\left\{
f \in {\mathcal S}'/{\mathcal P}\,:\,
\|f\|_{\dot{F}^s_{pq}}
\equiv
\left\|
\left(
\sum_{j=-\infty}^\infty 
|2^{j s}{\mathcal F}^{-1}[\varphi_j \cdot {\mathcal F}f]|^q
\right)^{\frac{1}{q}}
\right\|_p<\infty
\right\}.
\]
The space
$\dot{F}^s_{pq}$ is the set of all
$f \in {\mathcal S}'/{\mathcal P}$
for which the quasi-norm
$\|f\|_{\dot{F}^s_{pq}}$
is finite.
\end{definition}
To handle the convolution,
we use the following theorem:
\begin{theorem}\label{thm:160118-1}
Let $0<p<\infty$, $0<q \le \infty$ and $s \in {\mathbb R}$.
If $0<\eta<\infty$ and
$\{f_j\}_{j=1}^\infty$ is a sequence of measurable functions,
then 
\[
\left\| \left(\sum_{j=1}^\infty Mf_j{}^q\right)^{\frac1q}
\, \right\|_p
\lesssim_{p,q,\eta}
\left\| \left(\sum_{j=1}^\infty |f_j|^q\right)^{\frac1q}
\, \right\|_p.
\]
\end{theorem}

\subsection{The nonhomogeneous Besov space $B^s_{pq}$ for $0<p,q \le \infty$ and $s \in {\mathbb R}$}
\label{section:6.3}

Choose $\psi \in {\mathcal S}$
so that
\begin{equation}\label{eq:151224-1}
\chi_{B(2)} \le \psi \le \chi_{B(4)}.
\end{equation}
Define
\begin{equation}\label{eq:151224-2}
\varphi_j \equiv \psi(2^{-j}\cdot)-\psi(2^{-j+1}\cdot)
\end{equation}
for $j \in {\mathbb N}$.

Using (\ref{eq:151224-1}) and (\ref{eq:151224-2}),
we define the {\it nonhomogeneous Besov space}
$B^s_{pq}$
as follows:
\begin{definition}
Let $0<p,q \le \infty$ and $s \in {\mathbb R}$.
Define
\begin{align*}
\lefteqn{
B^s_{pq}
}\\
&\equiv
\left\{
f \in {\mathcal S}'\,:\,
\|f\|_{B^s_{pq}}
\equiv
\|{\mathcal F}^{-1}[\psi \cdot {\mathcal F}f]\|_p
+
\left(
\sum_{j=1}^\infty 
(2^{j s}\|{\mathcal F}^{-1}[\varphi_j \cdot {\mathcal F}f]\|_p)^q
\right)^{\frac{1}{q}}<\infty
\right\}.
\end{align*}
The {\it nonhomogeneous Besov space}
$B^s_{pq}$ is the set of all
$f \in {\mathcal S}'$
for which the quasi-norm
$\|f\|_{B^s_{pq}}$
is finite.
\end{definition}

Here we collect some fundamental properties
of Besov spaces,
which we prove as we did in this note.
\begin{remark}
\
\begin{enumerate}
\item
The definition of
$B^s_{pq}$ does not depend on 
$\psi$ satisfying
(\ref{eq:151224-1}).
\item
$B^s_{pq}$ is a complete space,
that is,
$B^s_{pq}$ is a quasi-Banach space
when
$0<p,q \le \infty$
and
$B^s_{pq}$ is a Banach space
when
$1\le p,q \le \infty$.
\item
${\mathcal S} 
\hookrightarrow
B^s_{pq}
\hookrightarrow
{\mathcal S}'$
in the sense of continuous embedding.
\item
$B^0_{22}=L^2$ with norm equivalence.
\item
Define ${\mathcal C}^s \equiv \dot{\mathcal C}^s \cap L^\infty$,
where $\dot{\mathcal C}^s$ is NOT a linear space
modulo ${\mathcal P}_{[s]}$.
Then,
${\mathcal C}^s$ is a Banach space
and
$B^s_{\infty \infty}={\mathcal C}^s$
for $s>0$ with norm equivalence.
\item
Let $0<p_0<p_1<\infty$, $0<q \le \infty$ and
$-\infty<s_1<s_0<\infty$.
Assume
\[
s_0-\frac{n}{p_0}=s_1-\frac{n}{p_1}.
\]
Then
$B^{s_0}_{p_0q} \hookrightarrow B^{s_1}_{p_1q}$.
\end{enumerate}
\end{remark}

\subsection{The nonhomogeneous Triebel-Lizorkin space $F^s_{pq}$ for $0<p<\infty$, $0<q \le \infty$ and $s \in {\mathbb R}$}
\label{section:6.4}

Let $\psi$ and $\varphi$
satisfy
(\ref{eq:151224-1})
and
(\ref{eq:151224-2}),
respectively.

\begin{definition}
Let $0<p,q \le \infty$ and $s \in {\mathbb R}$.
Define
\begin{align*}
\lefteqn{
F^s_{pq}
}\\
&\equiv
\left\{
f \in {\mathcal S}'\,:\,
\|f\|_{F^s_{pq}}
\equiv
\|{\mathcal F}^{-1}[\psi \cdot {\mathcal F}f]\|_p
+
\left\|
\left(
\sum_{j=1}^\infty 
|2^{j s}{\mathcal F}^{-1}[\varphi_j \cdot {\mathcal F}f]|^q
\right)^{\frac{1}{q}}
\right\|_p<\infty
\right\}.
\end{align*}
The {\it nonhomogeneous Triebel-Lizorkin space}
$F^s_{pq}$ is the set of all
$f \in {\mathcal S}'$
for which the quasi-norm
$\|f\|_{F^s_{pq}}$
is finite.
\end{definition}

Here we collect some fundamental properties
of Triebel-Lizorkin spaces,
whose proof overlaps largely
the ones in this note.
Here we comment
what else idea we need if necessary.
\begin{remark}
\
\begin{enumerate}
\item
The definition of
$F^s_{pq}$ does not depend on 
$\psi$ satisfying
(\ref{eq:151224-1}).
\item
$F^s_{pq}$ is a complete space,
that is,
$F^s_{pq}$ is a quasi-Banach space
when
$0<p,q \le \infty$
and
$F^s_{pq}$ is a Banach space
when
$1\le p,q \le \infty$.
\item
${\mathcal S} 
\hookrightarrow
B^s_{p\min(p,q)}
\hookrightarrow
F^s_{pq}
\hookrightarrow
B^s_{p\max(p,q)}
\hookrightarrow
{\mathcal S}'$
in the sense of continuous embedding.
\item
Let $1<p<\infty$.
Then
$F^0_{p2}=L^p$ with norm equivalence.
This is proved by using the {\it Rademacher sequence}.
\item
Let $0<p_0<p_1<\infty$, $0<q \le \infty$ and
$-\infty<s_1<s_0<\infty$.
Assume
\[
s_0-\frac{n}{p_0}=s_1-\frac{n}{p_1}.
\]
Then
$F^{s_0}_{p_0\infty} \hookrightarrow F^{s_1}_{p_1q}$.
\item
Let $0<p_0<p_1<\infty$, $0<q \le \infty$ and
$-\infty<s_1<s_0<\infty$.
Assume
\[
s_0-\frac{n}{p_0}=s_1-\frac{n}{p_1}.
\]
Then
$B^{s_0}_{p_0p_1} \hookrightarrow F^{s_1}_{p_1q}$
and
$F^{s_0}_{p_0\infty} \hookrightarrow B^{s_1}_{p_1p_0}$.
\cite{Franke86,Jawerth77,Vybiral08}
\end{enumerate}
\end{remark}

\section{Appendix}

\subsection{Proof of Theorem \ref{thm:151224-1}}
\label{Section:A3}

Let $\psi \in {\mathcal S}$ satisfy
$(\ref{eq:151224-21})$.
Define
$\varphi_j$
by $(\ref{eq:151224-22})$ for $j \in {\mathbb Z}$.
We adopt the following definition of the Besov norm:
\[
\|f\|_{\dot{B}^s_{pq}}
\equiv
\left(
\sum_{j=-\infty}^\infty 
(2^{j s}\|{\mathcal F}^{-1}[\varphi_j \cdot {\mathcal F}f]\|_p)^q
\right)^{\frac{1}{q}}
\]
for $f \in {\mathcal S}'/{\mathcal P}$.
Set
\[
\psi_{j,k}(\xi) \equiv
\begin{cases}
\xi_k|\xi|^{-1}
(\varphi_{j-1}(\xi)+\varphi_j(\xi)+\varphi_{j+1}(\xi))
&(\xi \in {\mathbb R}^n \setminus \{0\}),\\
0&(\xi=0).
\end{cases}
\]
Then arguing as before, we have
\begin{align*}
\|R_k f\|_{\dot{B}^s_{pq}}
&=\left(
\sum_{j=-\infty}^\infty 
(2^{j s}\|{\mathcal F}^{-1}
[\xi_k|\xi|^{-1}\varphi_j \cdot {\mathcal F}f]\|_p)^q
\right)^{\frac{1}{q}}\\
&=\left(
\sum_{j=-\infty}^\infty 
(2^{j s}\|{\mathcal F}^{-1}
[\psi_{j,k}\varphi_j \cdot {\mathcal F}f]\|_p)^q
\right)^{\frac{1}{q}}\\
&\simeq\left(
\sum_{j=-\infty}^\infty 
(2^{j s}\|{\mathcal F}^{-1}\psi_{j,k}
*{\mathcal F}^{-1}[\varphi_j \cdot {\mathcal F}f]\|_p)^q
\right)^{\frac{1}{q}}\\
&\le \left(
\sum_{j=-\infty}^\infty 
(2^{j s}\|{\mathcal F}^{-1}\psi_{j,k}\|_1
\|{\mathcal F}^{-1}[\varphi_j \cdot {\mathcal F}f]\|_p)^q
\right)^{\frac{1}{q}}\\
&\lesssim
\|f\|_{\dot{B}^s_{pq}},
\end{align*}
as was to be shown.

\subsection{Proof of Lemma \ref{lem:151215-1}}
\label{Section:A2}

We prove this lemma by induction on $N$.
If $N=1$, then the result is immediate from the Poincar\'{e}
lemma, or equivalently, $H^1_{\rm DR}({\mathbb R}^n)=0$.
Let $N \ge 2$.
Define
$g_{j\gamma} \equiv f_{e_j+\gamma}$
if $j=1,2,\ldots,n$ and $|\gamma|=N-1$.
Let $e_j \equiv (0,\ldots,0,1,0,\ldots,0)$,
where $1$ is in the $j$-th lot.
Then we have
\[
\partial^{\delta}g_{j\gamma}
=
\partial^{\delta}f_{e_j+\gamma}
=
\partial^{\delta'}f_{e_{j'}+\gamma}
=
\partial^{\delta'}g_{j'\gamma}
\]
for all 
$j,j'=1,2,\ldots,n$,
$\gamma,\delta,\delta' \in {\mathbb N}_0{}^n$
with $|\gamma|=N-1$ and
$e_j+\delta=e_{j'}+\delta'$.
Thus we are in the position of applying
the Poincar\'{e} lemma to have
$g_j \in C^\infty$ satisfying
$g_{j\gamma}=\partial^j g_\gamma$
for all
$\gamma$ with $|\gamma|=N-1$
and $j=1,2,\ldots,n$.

Let $\alpha,\alpha',\gamma,\gamma' \in {\mathbb N}$
satisfy
$|\gamma|=|\gamma'|=N-1$
and
$|\alpha|=|\alpha|>0$
and
$\alpha+\gamma=\alpha'+\gamma'$.
Suppose $\alpha-e_j,\alpha'-e_{j'} \ge 0$.
Then we have
\[
\partial^{\alpha}g_{\gamma}
=
\partial^{\alpha-e_j}\partial^j g_{\gamma}
=
\partial^{\alpha-e_j}g_{j\gamma}
=
\partial^{\alpha-e_j}f_{e_j+\gamma}
\]
and
\[
\partial^{\alpha'}g_{\gamma'}
=
\partial^{\alpha'-e_{j'}}\partial^{j'} g_{\gamma'}
=
\partial^{\alpha'-e_{j'}}g_{j'\gamma'}
=
\partial^{\alpha'-e_{j'}}f_{e_{j'}+\gamma'}.
\]
Since
$(\alpha-e_j)+(e_j+\gamma)=(\alpha'-e_{j'})+(e_{j'}+\gamma')$,
we obtain
$\partial^{\alpha}g_{\gamma}
=\partial^{\alpha'}g_{\gamma'}.
$
Thus by the induction assumption to have a function
$f \in C^\infty$ such that
$g_{\gamma}=\partial^\gamma f$.
Consequently,
if $|\alpha|=N$,
then by choosing $j$ so that $e_j \le \alpha$
we have
$f_\alpha=\partial_j g_{\alpha-e_j}
=\partial_j\partial^{\alpha-e_j}f
=\partial^\alpha f$.

\section{Historical notes}

\subsection{Function spaces}

Let us go back to the study of Sobolev spaces
initiated in
\cite{Sobolev35,Sobolev36,Sobolev38}.
Zygmund found out that
the difference of the second order is useful
in \cite{Zygmund45}.
The space
$W^s_p$ with $s \in (0,\infty) \setminus {\mathbb N}$
and $1 \le p \le \infty$
dates back to
\cite{Aronszajn55,Slobodeckij58,Gagliardo58}.
The space
$\Lambda_{p,\infty}^s$
is defined by Besov
in \cite{Besov59,Besov61}.
Let $f \in L^p$ with $1 \le p<\infty$. 
When $0<s<1$,
the norm 
$\|f\|_{\Lambda^s_{p,\infty}}$
is given by
\[
\|f\|_{\Lambda^s_{p,\infty}}
\equiv
\|f\|_{L^p}
+
\sup_{h \in {\mathbb R}^n \setminus \{0\}}
|h|^{-s}\|f(\cdot+h)-2f+f(\cdot-h)\|_p.
\]
When $s=1$,
the norm 
$\|f\|_{\Lambda^s_{p,\infty}}$
is given by
\[
\|f\|_{\Lambda^s_{p,\infty}}
\equiv
\|f\|_{W^1_p}
+
\sup_{h \in {\mathbb R}^n \setminus \{0\}}
\|\nabla[f(\cdot+h)-2f+f(\cdot-h)]\|_p.
\]
When $1<s<2$,
the norm 
$\|f\|_{\Lambda^s_{p,\infty}}$
is given by
\[
\|f\|_{\Lambda^s_{p,\infty}}
\equiv
\|f\|_{W^1_p}
+
\sup_{h \in {\mathbb R}^n \setminus \{0\}}
|h|^{-s+1}
\|\nabla[f(\cdot+h)-2f+f(\cdot-h)]\|_p.
\]
Let $1 \le p \le \infty$ and $s>0$.
The space 
$H^s_p$
was introduced by Aronszajn and Smith
in \cite{ArSm61}
and
by Calder\'{o}n
in \cite{Calderon61}.

\subsection{Theorems \ref{thm:2} and \ref{thm:151224-109}}

Taibleson, Rivi\`{e}re and Sagher 
proved that
${\mathcal F}$ maps
$B^0_{1\infty}$ to $L^\infty$;
see \cite{RiSa74,Taibleson64}.
See also \cite[p. 116 (3)]{Peetre-text-76}
for the fact that the Fourier transform
maps $\dot{B}^0_{1\infty}$,
the homogeneous Besov space,
into $L^\infty$.
See \cite[Chapter 6]{Peetre-text-76}
for various extensions of this boundedness.
Gabisoniya proved 
${\mathcal F}$ maps
$B^{n/2}_{21}$ to $L^1$
assuming $f \in L^2$ \cite{Gabisoniya66}.
On the torus Bernstein consdered
a counterpart to the Fourier series. 
Golovkin and Solonnikov proved the results
assuming that both sides are finite in \cite[Section 3, Theorem 7]{GoSo68}.
Madych gave a sharper version of this fact
in \cite{Madych76}.

\subsection{Theorem \ref{thm:151224-1}}

The boundedness of the Riesz transform
on $\dot{B}^s_{pq}$ is an immediate consequence
of \cite[5.2.2]{Triebel-text-83}.

\subsection{Theorem \ref{thm:160118-5}}

The equivalence
(\ref{eq:160118-1}) can be found in 
 \cite{LiPa31,LiPa37-1,LiPa37-2}
and it is the beginning of 
the Littlewood and Paley theory.
The equivalence
(\ref{eq:160118-2}), which asserts
$H^p \approx \dot{F}_{p2}^0$, 
dates back to Peetre
\cite{Peetre73,Peetre-text-74}.
Meanwhile,
the equivalence
(\ref{eq:160118-3}), which asserts
$h^p \approx F_{p2}^0$, is due to
\cite{Bui80}.
This equivalence motivates the definition
of the space $F^s_{pq}$.

(\ref{eq:160118-2})
and
(\ref{eq:160118-3})
date back to Goldberg \cite{Goldberg79-1,Goldberg79-2}.
See also 
\cite[p. 93, Remark]{Triebel-text-83}
and
\cite[p. 92, Theorem]{Triebel-text-83}
for the proof of 
(\ref{eq:160118-3})
 and
(\ref{eq:160118-2}),
respectively.
We can find (\ref{eq:160118-4})
in \cite[Remark 3]{Triebel-text-83}.
In \cite{Goldberg79-1,Goldberg79-2},
Goldberg investigated
local ${\rm bmo}$ space as well as local Hardy spaces.
(\ref{eq:160118-5}) is due to \cite[p. 36]{Goldberg79-1}.

\subsection{Theorem \ref{thm:160118-4}}

Theorem \ref{thm:160118-4} is due to 
Schwartz \cite[p. 105, Th\'{e}or\`{e}me XII]{Schwartz-text}.

\subsection{Theorem \ref{thm:151209-1}}
We can find a detailed proof of Theorem \ref{thm:151209-1}
(1)--(3) in \cite[Proposition 8.1]{YSY10-2}.
We refer to \cite[Section 6]{NNS15}
for the proof of (4).

\subsection{Theorem \ref{thm:151211-1}}

The author was not sure
when Theorem \ref{thm:151211-1} initially appeared.
To the best knowledge of the author,
we can find it in the textbook
\cite{BeLo-text-76}
without proof.

\subsection{H\"{o}lder-Zygmund spaces}

See \cite{Zygmund45} for H\"{o}lder-Zygmund spaces.

\subsection{Theorems \ref{thm:151209-9} and  \ref{thm:151209-2}}

We refer to \cite[Section 6]{NNS15}
for the proof.

\subsection{Theorem \ref{thm:160118-3}}

See the textbooks
\cite{BeLo-text-76,Peetre-text-76}
for the definition of the Besov space
$B^s_{pq}$
with $1 \le p,q \le \infty$ and $s \in {\mathbb R}^n$
as in this book.
We can find the definition of 
the inhomogeneous space $B^s_{p q}$
in \cite[Definition 6.2.2]{BeLo-text-76}
and \cite[p. 48, Definition 1]{Peetre-text-76}.
For the definition of homogeneous Besov spaces
see \cite[6.3]{BeLo-text-76}.

\subsection{Theorem \ref{thm:151224-102}}

Theorem
\ref{thm:151224-102}
is essentially due to Taibleson \cite[Theorem 4]{Taibleson64}.
See also \cite[p. 90, Theorem]{Triebel-text-83}.
For $s>2$,
${\mathcal C}^s$ is a function space
applied to elliptic differential equations
in the paper \cite{ADN59}.
See the textbook of Miranda
\cite{Miranda-text-70}
for the detailed background.
We refer to
\cite{EdGa-text-77,Grevholm70,Triebel78-2}
for further details.

\subsection{Theorems \ref{thm:151210-7} and \ref{thm:160118-2}}

We can find 
Theorems \ref{thm:151210-7} and \ref{thm:160118-2}
in \cite[p. 240, Theorem]{Triebel-text-83}.

\subsection{Theorem \ref{thm:151224-103}}

We can find Theorem \ref{thm:151224-103}
in \cite[5.2.3]{Triebel-text-83}.

\subsection{Theorem \ref{thm:151224-101}}

The embedding (\ref{eq:160119-1})
dates back to  Triebel \cite{Triebel73-1} 
when $1<p,q<\infty$.
See also \cite[Theorem 6.5.1]{BeLo-text-76}.
For the general case
(\ref{eq:160119-1}) is due
to Jawerth \cite{Jawerth77}.
See also \cite[p. 129, Theorem]{Triebel-text-83}.

\subsection{Theorem \ref{thm:151215-111}}

We refer to the paper \cite{Bourdaud11}
for more about the case for general $s$.

\subsection{Theorem \ref{thm:160118-1}}

We refer to \cite{FeSt71}
for Theorem \ref{thm:160118-1}.
See also
\cite{GaRu-text-85,Stein-text-93}.

\subsection{Definition of function spaces}

There is a long history in the definition of the function spaces.
So the proof is a little simpler.
Triebel used the Fourier multiplier
very systematically
in \cite[Theorem 3.5]{Triebel73-1}
to define $B^s_{pq}$
with $1<p,q<\infty$ and $s \in {\mathbb R}$.
We refer to \cite[p. 225--227]{Peetre-text-76}
and \cite{Peetre73}
for the motivation of function spaces $B^s_{pq}$
with $0<p<1$.
For the definition of $B^s_{pq}$
with $1\le p \le \infty$, $0<q \le \infty$ and $s \in {\mathbb R}$
with $0< p \le \infty$, $0<q \le \infty$ and $s \in {\mathbb R}$
we refer to 
\cite[p. 48, Definition 1]{Peetre-text-76}
and
\cite[p. 232, Definition 1]{Peetre-text-76}.
In his 1976 book \cite{Peetre-text-76}
Peetre showed that the definition
of $\dot{B}^s_{pq}$ is independent
of the choice of $\varphi$
whose proof in that book is similar to the one in this book.

\subsection{Textbooks on Besov spaces}

We list
\cite{Burenkov-text-54,EdTr-text-96,EdGa-text-77,Frazier-text-87,
Grafakos-text-09,H,HaTr-text-08,RuSi-text-96,Skrzypczak-text-91,Triebel-text-92,Triebel-text-97,
Triebel-text-00,Triebel-text-11,Triebel-text-12-1,Triebel-text-12-2,Triebel-text-13,Triebel-text-14}
as textbooks of function spaces.

\end{document}